\documentclass[sts]{imsart}

\usepackage{amssymb,amsbsy,amsmath,amscd,amsthm,amsfonts,MnSymbol}
\usepackage{cite}
\usepackage{enumerate}
\usepackage{dsfont}
\usepackage{graphicx}

\usepackage[colorlinks,citecolor=blue,urlcolor=blue]{hyperref}
\usepackage{xcolor,colortbl}
\usepackage{cleveref, enumitem}

\hypersetup{hidelinks}

\newtheorem{definition}{Definition}
\newtheorem{theorem}{Theorem}
\newtheorem{assumption}{Assumption}

\newtheorem{lemma}{Lemma}
\newtheorem{remark}{Remark}

\newcommand{\R}{\mathbb R}
\newcommand{\PP}{\mathbb P}
\newcommand{\E}{\mathbb E}
\newcommand{\K}{\mathcal K_d}

\newcommand{\Sp}{\mathbb S^{d-1}}
\newcommand{\KK}{\mathcal K_d^{(1)}}

\newcommand*\diff{\mathop{}\!\mathrm{d}}

\newcommand{\DS}{\displaystyle}

\startlocaldefs
\endlocaldefs

\begin{document}

\begin{frontmatter}

% "Title of the paper"
\title{Methods for Estimation of Convex Sets}
\runtitle{Estimation of Convex Sets}

% indicate corresponding author with \corref{}
% \author{\fnms{John} \snm{Smith}\corref{}\ead[label=e1]{smith@foo.com}\thanksref{t1}}
% \thankstext{t1}{Thanks to somebody} 
% \address{line 1\\ line 2\\ printead{e1}}
% \affiliation{Some University}

\author{\fnms{Victor-Emmanuel} \snm{Brunel}\ead[label=e1]{vebrunel@mit.edu}}
\address{\printead{e1}}
\affiliation{Massachusetts Institute of Technology, Department of Mathematics}
%\and
%\author{\fnms{???} \snm{???}\ead[label=e2]{???}}
%\address{\printead{e2}}
%\affiliation{???}

\runauthor{V.-E. Brunel}

\begin{abstract}

In the framework of shape constrained estimation, we review methods and works done in convex set estimation. 
These methods mostly build on stochastic and convex geometry, empirical process theory, functional analysis, linear programming, extreme value theory, etc. The statistical problems that we review include density support estimation, estimation of the level sets of densities or depth functions, nonparametric regression, etc. We focus on the estimation of convex sets under the Nikodym and Hausdorff metrics, which require different techniques and, quite surprisingly, lead to very different results, in particular in density support estimation. Finally, we discuss computational issues in high dimensions.

\end{abstract}



\end{frontmatter}

\section{Preliminaries}\label{Sec:Prelim}

\subsection{Introduction}

In nonparametric inference, the unknown object of interest cannot be described in terms of a finite number of parameters. Examples include density estimation, nonparametric and high dimensional regression, support estimation, etc. Since the number of observations is only finite, it is necessary to make assumptions on the object of interest in order to make statistical inference significant. Two types of assumptions are most common: Smoothness assumptions and shape constraints. A smoothness assumption usually imposes differentiability up to some fixed order, with bounded derivatives (the reader could find an introduction to the estimation of smooth density or regression functions in \cite[Chapter 1]{Tsybakov2009}; \cite{MammenTsybakov1995} imposes smoothness assumptions on the boundary of the support of an unknown density or on the boundary of an unknown set in image reconstruction from random observations). Shape constraints rather impose conditions such as monotonicity, convexity, log-concavity, etc. (e.g., \cite{kim2016global} assumes log-concavity of the unknown density; \cite{ChatterjeeAl2015} imposes monotonicity or more general shape constraints on the unknown regression function; \cite{KorostelevSimarTsybakov1995'} imposes a monotonicity or a convexity constraint on the boundary of the support of the unknown density; \cite{KorostelevTsybakov1994,Brunel2016} impose convexity on the support of the unknown distribution).

Smoothness is a quantitative condition, whereas a shape constraint is usually qualitative. Smoothness classes of functions or sets depend on meta parameters, such as the number of existing derivatives or upper bounds on some functional norms. However, in statistical applications, these meta parameters are unlikely to be known to the practitioner. Yet, statistical inference usually requires to choose tuning parameters that depend on these meta parameters. One way to overcome this issue is to randomize the tuning parameters and apply data driven adaptive procedures such as cross validation. However, such procedures are often technical and computationally costly. On the opposite, shape constraints usually do not introduce extra parameters, which makes them particularly attractive.

Many different shape constraints can be imposed on sets. For instance, \cite{Tsybakov1994,KorostelevSimarTsybakov1995,KorostelevSimarTsybakov1995'} consider boundary fragments, which are the subgraphs of positive functions defined on a hypercube (or, more generally, on a metric space). Shape constraints on such sets directly translate into shape constraints on their edge functions. For general sets, convexity is probably the most simple shape constraint, even though it leads to a very rich field in geometry. Convexity can be extended to the notion of $r$-convexity, where an $r$-convex set is the complement of the union of open Euclidan balls of radius $r$, $r>0$ (see, e.g., \cite{ManiLevitska1993} and \cite{Rodriguez2007,Pateiro2008} for set estimation under $r$-convexity and, more generally, \cite{Cuevas2009,cuevas2012statistical} for broader shape constraints in set estimation). Informally, convexity is the limit of $r$-convexity as $r$ goes to $\infty$. In set estimation, if it is assumed that the unknown set is $r$-convex for some $r>0$, the meta parameter $r$ may also be unknown to the practitioner and \cite{Rodriguez2016} defines a data-driven procedure that adapts to $r$. In the present article, we only focus on convexity, which is a widely treated shape constraint in statistics. On top of convexity, two additional constraints are common in statistics: the rolling ball condition and standardness. A convex set $G$ is said to satisfy the $r$-rolling ball condition ($r>0$) if, for all $x$ on the boundary of $G$, there is a Euclidean ball $B(a,r)$ such that $x\in B(a,r)\subseteq \bar G$, where $\bar G$ is the closure of $G$ (see \cite{Walther1997,Walther1999} for characterizations of the rolling ball condition, connections with $r$-convexity and statistical applications in set estimation). An equivalent condition is that the complement of $G$ has reach at least $r$. The reach of a set is the supremum of all positive numbers $\varepsilon$ such that any point within a distance $\varepsilon$ of that set has a unique metric projection onto the closure of that set (see \cite[Definition 11]{Thale2008}). A convex set $G$ is called $\nu$-standard ($\nu\in (0,1)$) if for all $x$ on its boundary, $\textsf{Vol}\left(G\cap B(x,\varepsilon)\right)\geq \nu\textsf{Vol}\left(B(x,\varepsilon)\right)$, for all $\varepsilon>0$ small enough. This roughly means that the set $G$ does not have peaks. 

In general, two main types of convex bodies are distinguished in the literature. 
\begin{itemize}
	\item Convex bodies with smooth boundary: The boundary $\partial G$ of a convex body $G$ is \textit{smooth} if for all $x\in\partial G$, $G$ has a unique supporting hyperplane that contains $x$. In that case, let $H_x$ be the unique supporting hyperplane containing $x$ and let $\eta_x$ be the unit vector orthogonal to $H_x$ and pointing towards the inside of $G$. Identify the $(d-1)$-dimensional linear subspace $H_x-x=\{z-x:z\in H_x\}$ with $\R^{d-1}$; Then, every $y\in\partial G$ that is in some neighborhood of $x$ can be written uniquely as $y=x+t+f_x(t)\eta_x$, where $t\in H_x-x$ and $f_x$ is a nonnegative convex function defined in a neighborhood of $0$ in $H_x-x$. If the Hessian of $f_x$ at $0$ is positive definite, $\partial G$ is said to have \textit{positive curvature} at $x$. Otherwise, $\partial G$ has \textit{zero curvature} at $x$. 
	\item Convex polytopes: A convex polytope (in short, a polytope) is the convex hull of finitely many points in $\R^d$. By the Minkowski--Weyl theorem, a polytope can also be represented as the intersection of finitely many closed halfspaces. The supporting hyperplane of a polytope $P$ containing $x\in\partial P$ is unique if $x$ is not in a $k$-dimensional face of $P$ for some $k\leq d-2$, and $P$ has zero curvature at all such boundary points $x$.
\end{itemize}
We refer the readers who are interested in learning more about convex bodies to \cite{Schneider1993}, and to \cite{Ziegler1995} for a comprehensive study of convex polytopes.

In the field of nonparametric statistics, the problem of set estimation arose essentially with the works \cite{Geffroy1964} (\textit{On a geometric estimation problem}) and \cite{Chevalier1976}, which deal with the estimation of the support of a density in a general setup. A simple and natural estimator of the support of an unknown density was introduced in \cite{DevroyeWise1980}, where the estimator is defined as the union of small Euclidean balls centered around the data points. In fact, this estimator is equal to the support of a kernel density estimator for the kernel that is the indicator function of the Euclidean unit ball.

The scope of this survey is the estimation of convex sets. We aim to give an exposition of several methods that build on stochastic and convex geometry, empirical process theory, functional analysis, linear programming, order statistics and extreme value theory, etc. Different models associated with the estimation of convex sets include density support estimation \cite{KorostelevTsybakov1993,KorostelevSimarTsybakov1995,KorostelevSimarTsybakov1995',Brunel2016,BrunelHausdorff2017,Brunel2017
}, density level set estimation \cite{Hartigan1987,Polonik1995,Tsybakov1997}, inverse problems in density support estimation \cite{BrunelKlusowski2017}, estimation of the support of a regression function \cite{KTlectureNotes1993,Tsybakov1994,Brunel2013}, estimation of the level sets of the Tukey depth function \cite{BrunelTDLS2017}, estimation of support functions \cite{GardnerKiderlenMilanfar2006,Guntuboyina2012}, etc.

Throughout this survey, a \textit{set estimator} is a set-valued statistic, i.e., a set which depends on the observed random variables. A precise definition would be necessary in order to rule out measurability issues. However, in order to keep the focus on convex set estimation, we rather choose not to mention these issues, and all probabilities (resp. expectations) should be understood as outer probabilities (resp. expectations). For detailed accounts on set-valued random variables, we refer to \cite{Molchanov2005}.

Before going more into the details, let us introduce some notation and definitions.

\subsection{Notation and Definitions}

In the sequel, $d$ is a positive integer, standing for the ambient dimension. For a positive integer $p$, the closed $p$-dimensional Euclidean ball with center $a\in\R^p$ and radius $r\geq 0$ is denoted by $B_p(a,r)$. If $p=d$, we may omit the subscript $p$. The $(p-1)$-dimensional unit sphere is denoted by $\mathbb S^{p-1}$ and the volume of the $p$-dimensional unit Euclidean ball is denoted by $\beta_p$. The Euclidean norm in $\R^d$ is denoted by $\|\cdot\|$, the Euclidean distance is $\rho$ and we write $\langle\cdot,\cdot\rangle$ for the canonical dot product.

A convex body $G\subseteq \R^d$ is a compact and convex set with nonempty interior. We denote by $\mathcal K_d$ the collection of all convex bodies and by $\mathcal K_d^{(1)}$ the collection of all convex bodies included in $B(0,1)$. The support function $h_G$ of a convex body $G$ is defined as $h_G(u)=\max\{\langle u,x\rangle:x\in G\}$, for all $u\in\Sp$: It is the signed distance of tangent hyperplanes to the origin.

The volume of a measurable set $A\subseteq\R^d$ is denoted by $|A|$.

The Nikodym distance between two measurable sets $K,L\subseteq\R^d$ is the volume of their symmetric difference: $\textsf{d}_\triangle(K,L)=|K\triangle L|$. The Hausdorff distance between any two sets $K,L\subseteq\R^d$ is defined as $\textsf{d}_{\textsf{H}}(K,L)=\inf\{\varepsilon\geq 0: K\subseteq L+\varepsilon B(0,1), L\subseteq K+\varepsilon B(0,1)\}$.

The cardinality of a finite set $A$ is denoted by $\#A$. 

When i.i.d. random points $X_1,\ldots,X_n\in\R^d$ have a density $f$ with respect to the Lebesgue measure, we denote by $\PP_f$ their joint distribution and by $\E_f$ the corresponding expectation operator, where we omit the dependency on $n$ for simplicity. When $f$ is the uniform density on a compact set $G$, we simply write $\PP_G$ and $\E_G$. 
The convex hull of $X_1,\ldots,X_n$ is denoted by $\hat K_n$. 

In this article, most, if not all, set-valued estimators are polytopes that depend on a finite random sample. Nonetheless, in order to be consistent with the literature, we reserve the name \textit{random polytope} for $\hat K_n$ only.

\subsection{Outline}

In order to assess the quality of a set estimator, the Nikodym and the Hausdorff metrics are most commonly used. Depending on which of these two metrics is to be used, the techniques in estimation of convex sets may differ a lot.

Section \ref{Sec:RandomPolytopes} is devoted to the estimation of convex sets under Nikodym-type metrics, especially in density support estimation. We first review essential properties of random polytopes and we relate them to the problem of support estimation under the Nikodym metric. We also recall well known results on the covering numbers of classes of convex bodies and show how these can be used in order to obtain deviation inequalities in convex support estimation. Then, we review extensions of these results to the estimation of density level sets under convexity and we discuss other convex set estimation problems under the Nikodym metric.

In Section \ref{Sec:SupportFunction}, we switch to the estimation of convex bodies under the Hausdorff metric. An elementary, yet essential result, stated in Lemma \ref{HausdorffSuppFunc}, shows that the Hausdorff distance between two convex bodies can be computed through their respective support functions. We review important properties related to the support functions of convex bodies and we show how they apply to the estimation of convex sets under the Hausdorff metric.

Finally, in Section \ref{Sec:CurseOfDimension}, we briefly discuss the computational aspects of convex set estimation in high dimensions. We show, through two examples, how to reduce the computational cost without affecting the rate of convergence of convex set estimators.

\section{Estimation of convex sets under the Nikodym metric} \label{Sec:RandomPolytopes}

\subsection{Random polytopes and density support estimation} \label{Subsec:RandomPol}

The most common representation of random polytopes consists of taking the convex hull of i.i.d. random points in $\R^d$. Stochastic and convex geometry have provided powerful tools to understand the properties of random polytopes, since the seminal works \cite{RenyiSulanke1963,RenyiSulanke1964}. In these two papers, $d=2$ and the random polygon is the convex hull of $n$ i.i.d. random points with the uniform distribution in a planar convex body. The expectation of the missing area and of the number of vertices of the random polygon are computed, up to negligible terms as $n$ goes to infinity. The results substantially depend on the structure of the boundary of the support. Namely, the expected missing area decreases significantly faster when the support is itself a polygon than when its boundary has positive curvature everywhere. The missing area is exactly the Nikodym distance between the random polygon and the support of the random points. Hence, \cite{RenyiSulanke1963,RenyiSulanke1964} give an approximate value of the risk of the random polygon as an estimator of the convex support. Later, much effort has been devoted to extend these results to higher dimensions, starting with \cite{Efron1965}, that proves integral formulas for the expected missing volume, surface area, number of vertices, etc. in dimension 3. Among most general results, a ground breaking one is due to \cite{BaranyLarman1988}. Define the \textit{$\varepsilon$-floating body} of a convex body $G\in\mathcal K_d$ as the set of points $x\in G$ such that any closed halfspace $H\subseteq \R^d$ containing $x$ has an intersection with $G$ whose volume is at least a fraction $\varepsilon$ of the total volume of $G$, i.e., satisfies $|G\cap H|\geq \varepsilon |G|$, where $\varepsilon\in (0,1)$ (see \cite{Dupin1822,Blaschke1923,SchuttWerner1990}). The \textit{$\varepsilon$-wet part} of $G$, denoted by $G(\varepsilon)$, is the complement of the $\varepsilon$-floating body of $G$ in $G$. If one thinks of $G$ as an iceberg seen from above, the floating body is the part of $G$ that is above the surface of the water, whereas the wet part is the immersed part of the iceberg.

\begin{theorem}[\cite{BaranyLarman1988}] \label{ThmBaranyLarman}

Let $G\in\mathcal K_d$ have volume one. Then,
$$c_1|G(1/n)|\leq \E_G[|G\setminus\hat K_n|]\leq c_2(d)|G(1/n)|, \quad \forall n\geq n_0(d),$$
where $c_1$ is a universal positive constant, $c_2(d)$ is a positive constant that depends on $d$ only and $n_0(d)$ is a positive integer that depends on $d$ only.

\end{theorem}

As a consequence, computing the expected missing volume of $\hat K_n$ asymptotically reduces to computing the volume of the $(1/n)$-wet part of $G$, which is no longer a probabilistic question. In addition, it is also known \cite{BaranyLarman1988} that if $G$ has volume one and $\varepsilon$ goes to zero, $|G(\varepsilon)|$ is of the order at least $\varepsilon\ln(1/\varepsilon)$ and at most $\varepsilon^{2/{d+1}}$ \cite{BaranyLarman1988}. The former rate is achieved when $G$ is a polytope whereas the latter rate is achieved when $G$ has a smooth boundary with positive curvature everywhere.

In fact, when $G$ is a smooth convex body with positive curvature everywhere, it is shown in \cite{Schutt1994} that 
\begin{equation} \label{ThmSchutt1994}
	n^{2/(d+1)}\E_G\left[\frac{|G\setminus\hat K_n|}{|G|}\right] \longrightarrow c(d,G), \hspace{4mm} n\to\infty,
\end{equation}
where $c(d,G)$ is an explicit positive constant that depends on $d$ and $G$ and that is affine invariant in $G$, i.e., $c(d,G)=c(d,T(G))$ for all invertible affine transormations $T$. \cite{Schutt1994} actually shows that this convergence holds for all convex bodies $G$, by noting that all convex bodies have a unique supporting hyperplane at almost all their boundary points (e.g., almost all boundary points of a polytope lie on a $(d-1)$-dimensional face), and where $c(d,G)$ is equal to zero if and only if $\partial G$ has zero curvature almost everywhere (e.g., if $G$ is a polytope).

An interesting result, due to \cite{Groemer1974}, shows that the quantity $\DS \E_G\left[\frac{|G\setminus\hat K_n|}{|G|}\right]$ is maximum when $G$ is an ellipsoid. In that case, it can be derived from \cite{Schutt1994} that the constant $c(d,G)$ in \eqref{ThmSchutt1994} is of the order $d^{d+o(d)}$, as $d$ becomes large. As a consequence, when the dimension $d$ becomes too large, the random polytope $\hat K_n$ performs poorly as an estimator of $G$ in the worst case, because it suffers the curse of dimensionality, both in the rate $n^{-2/(d+1)}$ and in the constant factor $d^{d+o(d)}$. Yet, it is known that the rate $n^{-2/(d+1)}$ cannot be improved in a minimax sense. The following result is proven in \cite{Brunel2016} where, for two sequences $a_n$ and $b_n$ of positive numbers, we write $a_n\lesssim_\theta b_n$ if $a_n\leq c(\theta)b_n, \forall n\geq 1$, for some positive constant $c(\theta)$ that depends on a parameter $\theta$. In the sequel, we also write $\lesssim$ with no subscript if the involved constant is universal.

\begin{theorem}[\cite{Brunel2016}] \label{ThmMinimaxSuppEst}

	The following inequalities hold:
$$n^{-\frac{2}{d+1}}\lesssim_d\inf_{\tilde G_n} \sup_{G\in\mathcal K^d}\E_G\left[\frac{|G\triangle\tilde G_n|}{|G|}\right]\leq \sup_{G\in\mathcal K^d}\E_G\left[\frac{|G\triangle\hat K_n|}{|G|}\right]\lesssim_d n^{-\frac{2}{d+1}},$$
where the infimum is taken over all estimators $\tilde G_n$ based on $n$ i.i.d. observations.

\end{theorem}

As a consequence, the random polytope $\hat K_n$ is rate optimal over the class $\mathcal K_d$ in a minimax sense, with respect to the Nikodym metric. The upper bound in Theorem \ref{ThmMinimaxSuppEst} is a direct consequence of Theorem \ref{ThmBaranyLarman}, together with Groemer's result \cite{Groemer1974}: It suffices to evaluate the volume of the $(1/n)$-wet part of a Euclidean ball of volume one. However, it is not clear that $\hat K_n$ is optimal in terms of the constant factors that become exponentially large with the dimension. Note that $\hat K_n\subseteq G$ with probability one, hence, $\hat K_n$ always underestimates the support $G$. This is why the estimation of $G$ through a dilation of $\hat K_n$ could be appealing. It has been considered, e.g., \cite{ripley1977finding} in the planar case for Poisson polytopes, and in \cite{Moore1984} for $d\leq 2$, but only heuristics are given in the general case, except for the estimation of the volume of $G$ in \cite{BaldinReiss2016}.  \cite{BaldinReiss2016} poses the question of the performance of a dilated version of $\hat K_n$ compared to that of $\hat K_n$ itself, but the question remains open.

Note that the lower bound in Theorem \ref{ThmMinimaxSuppEst} is also used in log-concave density estimation. The uniform density on any convex body is log-concave, and for any two convex bodies $G$ and $G'$ of volume $1$, the corresponding uniform densities $f_G$ and $f_{G'}$ satisfy $\|f_G-f_{G'}\|_2^2=|G\triangle G'|$, where $\|\cdot\|$ stands for the $L^2$ norm with respect to the Lebesgue measure in $\R^d$. Hence, some proof techniques for lower bounds on minimax risks in \cite{kim2016global} are based on similar arguments as those used to prove the lower bound in Theorem \ref{ThmMinimaxSuppEst}.

\subsection{Adaptation to polytopal supports} \label{Adapt}

As we already mentioned earlier, an attractive feature of most shape constraints is that no meta parameters are needed to describe the objects of interest, unlike in smoothness classes. 
Nonetheless, classes of functions or sets with a shape constraint usually contain parametric subclasses that correspond to simpler structures, which may depend on meta parameters. For instance, classes of monotone (resp. convex) functions contain piecewise constant (resp. affine) functions. A desirable property of an estimator is adaptation to these simpler structures: If the unknown object belongs to a subparametric class, then the rate of convergence of the estimator should be nearly as good as that of an estimator that would not be agnostic to that simpler structure. In recent years, there have been considerable efforts put in understanding this automatic adaptive features in shape constrained estimators \cite{KimGuntuboyinaSamworth2016,chatterjee2015adaptive,zhang2002risk,ChatterjeeAl2015,bellec2018sharp,han2017isotonic,han2016multivariate}.

Turning to the case of convex set estimation, the class $\K$ contains subclasses of polytopes with bounded number of vertices, hence, whose support functions are piecewise linear with a bounded number of pieces, each piece corresponding to a vertex. 

For the estimation of the convex support of a uniform distribution, the random polytope $\hat K_n$ is the maximum likelihood estimator on the class $\K$. Indeed, the likelihood function is given by $\DS L_n(C):=|C|^{-n}\prod_{i=1}^n\mathds 1_{X_i\in C}=|C|^{-n}\mathds 1_{\hat K_n\subseteq C}$, for all $C\in\K$ which is maximized when $C=\hat K_n$ (note that $\hat K_n\in\K$ with probability $1$ as long as $n\geq d+1$). Recall that, as a consequence of Theorem \ref{ThmMinimaxSuppEst}, in the Nikodym metric, $\hat K_n$ estimates $G$ at the speed $n^{-2/(d+1)}$ in the worst case, i.e., when $G$ has a smooth boundary. When $G$ is a polytope, Theorem \ref{ThmBaranyLarman} implies that $\hat K_n$ estimates $G$ at a much faster speed, namely, $n^{-1}(\ln n)^{d-1}$. A more refined (but not uniform in $G$) result was proven in \cite{BaranyBuchta1993}. For a polytope $P\subseteq\R^d$, let $T(P)$ be the number of flags of $P$, i.e., the number of increasing sequences $F_0\subseteq F_1\subseteq\ldots\subseteq F_{d-1}\subseteq F_d=P$ of faces of $P$ where $F_k$ is a $k$-dimensional face of $P$, $k=0,\ldots,d$. For example, $F(P)=2^d d!$ if $P$ is the $d$-dimensional hypercube, or $F(P)=d!$ if $P$ is the $(d-1)$-dimensional simplex.

\begin{theorem}[\cite{BaranyBuchta1993}]

	Let $G=P$ be a polytope. Then, 
$$\lim_{n\to\infty}\frac{n}{(\ln n)^{d-1}}\E_P\left[\frac{|P\setminus\hat K_n|}{|P|}\right]=\frac{T(P)}{(d+1)^{d-1}(d-1)!}.$$

\end{theorem}

In particular, if $G$ is a polytope, then there is a significant gain in the speed of convergence of $\hat K_n$, which becomes nearly parametric up to logarithmic factors. In other words, $\hat K_n$ adapts to polytopal supports. However, its rate still suffers the curse of dimensionality because of the $(\ln n)^{d-1}$ factor. In \cite{Brunel2016}, it is shown that this rate is not optimal over subclasses of polytopes with given number of vertices in a minimax sense. The idea is that $\hat K_n$ maximizes the likelihood function over the class of all convex bodies, which would too rich if it was known in advance that $G$ is a polytope with a given number of vertices. If $G$ has at most $r$ vertices, where $r\geq d+1$ is known \textit{a priori}, \cite{Brunel2016} considers the maximum likelihood estimator over the corresponding subclass of polytopes. Namely, denote by $\mathcal P_r$ the class of all polytopes with at most $r$ vertices. The maximum likelihood estimator of $P$ in the class $\mathcal P_r$ is defined as $\DS \hat P_n^{(r)}\in\underset{Q\in\mathcal P_r}{\operatorname{argmax}}\mbox{ } |P|^{-n}\mathds 1_{X_i\in P, \forall i=1,\ldots,n}$: It is a polytope with at most $r$ vertices that contains $X_1,\ldots,X_n$ and has minimum volume. Note that, unlike $\hat K_n$, the maximum likelihood estimator $\hat P_n^{(r)}$ may not be uniquely defined. However, the rate of this estimator no longer suffers the curse of dimensionality when $G\in\mathcal P_r$.

\begin{theorem}[\cite{Brunel2016}] \label{ThmPTRF16}
Let $r\geq d+1$. Then,
$$\frac{1}{n}\lesssim_d\inf_{\tilde G_n} \sup_{P\in\mathcal P_n^{(r)}}\E_P\left[\frac{|P\triangle\tilde G_n|}{|P|}\right]\leq \sup_{P\in\mathcal P_n^{(r)}}\E_P\left[\frac{|P\triangle\hat P_n^{(r)}|}{|P|}\right]\lesssim_d \frac{r\ln n}{n},$$
where the infimum is taken over all estimators $\tilde G_n$ based on $n$ i.i.d. observations.

\end{theorem}

In \cite{Brunel2016}, a better lower bound is proven when $d=2$, namely, $$\inf_{\tilde G_n} \sup_{P\in\mathcal P_n^{(r)}}\E_P\left[\frac{|P\triangle\tilde G_n|}{|P|}\right]\gtrsim \frac{r}{n}.$$

The proof of the upper bound in Theorem \ref{ThmPTRF16} builds on a simple discretization of the class $\mathcal P_n^{(r)}$, obtained by considering polytopes with vertices on a finite grid in $[0,1]^d$, and applying similar methods to those presented in Section \ref{Sec:CovNum}.

The estimator $\hat P_n^{(r)}$ is not computable in practice, but it gives a benchmark for the optimal rate in estimation of $P\in\mathcal P_r$, under the Nikodym metric. It is still not clear whether the logarithmic factor could be dropped in the upper bound (see \cite[Section 3.2]{Brunel2016}). A drawback of $\hat P_n^{(r)}$ is that it requires the knowledge of $r$, whereas $\hat K_n$ is completely agnostic to the facial structure of $G$. In order to fix this issue, \cite{Brunel2016} proposes a fully adaptive procedure and defines an estimator $\hat P_n^{\textsf{adapt}}$ that is agnostic of the facial structure of $G$ and yet performs at the same rate as $\hat P_n^{(r)}$ when $G\in\mathcal P_r$ for some integer $r\geq d+1$, and as $\hat K_n$ for general supports $G$ (see \cite{Brunel2016} and \cite{BrunelThesis} for more details). However, the estimators $\hat P_n^{(r)}$ and $\hat P_n^{\textsf{adapt}}$ are not computationally tractable, and when the dimension $d$ is not too large, the convex hull $\hat K_n$ is a more realistic estimator of $G$.

\subsection{More results on Random Polytopes}

Even though this survey focuses on the statistical aspects of random polytopes, it is worth mentioning many works that have tackled other probabilistic and geometric properties, which are indirectly related to the statistical estimation of the support and pose new statistical challenges.

In \cite{RenyiSulanke1963,RenyiSulanke1964}, the expected number of vertices of $\hat K_n$ is computed in the planar case, up to some negligible terms as $n\to\infty$. \cite{Efron1965} shows a very elegant identity which relates the missing volume of $\hat K_n$ and its number of vertices. It can be stated in a very general setup as follows. Given a sequence of i.i.d. random points $X_1,X_2,\ldots$ from some arbitrary probability measure $\mu$ in $\R^d$, let $\hat K_n$ be the convex hull of $X_1,\ldots,X_n$ and $N_{n}$ be the number of vertices of $\hat K_n$, for $n\geq 1$. Then, for all $n\geq 1$, 
$$\E\left[1-\mu(\hat K_n)\right]=\frac{\E[N_{n+1}]}{n+1}.$$
When $\mu$ is the uniform probability measure on a convex body $G\in\K$, this identity becomes $\DS \E_G\left[\frac{|G\setminus\hat K_n|}{|G|}\right]=\frac{\E_G[N_{n+1}]}{n+1}$. Extensions of this inequality to higher moments of $|G\setminus\hat K_n|$ can be found in \cite{buchta2005identity}.

In \cite{Reitzner2003}, more results about the random polytope $\hat K_n$, involving variance bounds, are proven using Efron--Stein jackknife inequalities \cite{efron1981jackknife}. Very importantly, \cite{Reitzner2003} compares the random polytope $\hat K_n$ to best polytopal approximations of smooth convex bodies. Let $G\in\K$ be a smooth convex body and let $G_N^*$ be a polytope with at most $N$ vertices, included in $G$, with minimum missing volume $|G\setminus G_N^*|$. With probability one, 
$$\frac{|G\setminus G_{N_n}^*|}{|G\setminus \hat K_n|}\longrightarrow c_d, \hspace{4mm} n\to\infty,$$
where $c_d\leq 1$ is a positive constant that only depends on the dimension $d$. Moreover, \cite{Reitzner2003} shows that $c_d\longrightarrow 1$ as $d\to\infty$. This shows that in high dimensions, with probability $1$, $\hat K_n$ performs nearly as well as the best approximating inscribed polytope with same number of vertices, as $n$ becomes large.

Central limit theorems for the volume, number of vertices, or, more generally, number of $k$-dimensional faces for $k\leq d-1$, of random polytopes are proven in \cite{reitzner2005central,Pardon2011,Pardon2012}. A worth mentioning technique that is used in the proofs of these central limit theorems could be called \textit{Poissonization-depoissonization}. The idea is to first consider a Poisson polytope, defined as the convex hull of a Poisson point process \cite{barany2010poisson} supported on a convex body, with growing intensity.  These are somewhat easier to work with, and it is shown that their behavior is close enough to that of the random polytope $\hat K_n$. Hence, the central limit theorems are first proven for the Poisson polytope, and the results are transferred to the random polytope by a \textit{depoissonization step}. At a high level, this idea relies on the fact that if $G\in\K$ has volume one and if $\mathcal X=\{X_1,\ldots,X_N\}$ is a Poisson point process with constant intensity $n$ supported on $G$, then $N$ is a Poisson random variable with parameter $n$, hence, $\E[N]=\textsf{Var}(N)=n$ and $N\approx n$ with high probability, and conditional on $N=n$, $X_1,\ldots,X_N$ are $n$ i.i.d. random points uniformly distributed in $G$.

Asymptotic properties of the intrinsic volumes of the random polytope are studied in \cite{Barany1992,Reitzner2004,BoroczkyHoffmannHug2008} under different assumptions on the boundary of the underlying convex body. The intrinsic volumes of a convex body can be defined through Steiner formula \cite[Section 4.1]{Schneider1993}. For $G\in\K$ and $\varepsilon>0$, let $G^{\varepsilon}=G+\varepsilon B(0,1)$ be the set of all points $x\in\R^d$ that are within a distance at most $\varepsilon$ of $G$. Steiner formula states that $|G^{\varepsilon}|$ is a degree $d$ polynomial in $\varepsilon$. Namely, one can write, for all $\varepsilon>0$, 
\begin{equation} \label{Steiner}
	|G^{\varepsilon}|=\sum_{j=0}^d \beta_{d-j}v_j(G)\varepsilon^j,
\end{equation}
where $v_j(G)\geq 0$ is called the $j$-th intrinsic volume of $G$, for $j=0,\ldots,d$. For instance, $v_0(G)=|G|$ is the volume of $G$, $v_1(G)$ is its surface area, $v_2(G)$ is its mean width and $v_d(G)=1$. In \cite{BoroczkyHoffmannHug2008}, it is shown that if $G$ is a smooth convex body satisfying the $r$-rolling ball condition, then for all $j=0,\ldots,d-1$, $\DS n^{2/(d+1)}\E_G[v_j(G)-v_j(\hat K_n)]\longrightarrow c(d,G)$ as $n\to\infty$, where $c(d,G)$ is a positive constant that depends on both the dimension and $G$. In particular, the plug-in estimator $v_j(\hat K_n)$ is a consistent estimator of $v_j(G)$, and it converges at the same rate as the rate of convergence of $\hat K_n$ in the Nikodym metric. Whether the plug-in estimator $v_j(\hat K_n)$ is an optimal estimator of $v_j(G)$ in a minimax sense is not known in general, except when $j=0$, when the answer is negative. \cite{Gayraud1997} considers the general problem of minimax estimation of the volume of the support of an unknown density, not necessarily uniform. In the particular case of the uniform density on an unknown convex body $G\in\K$, a sample splitting procedure is applied in order to correct the plug-in estimator $|\hat K_{n/2}|$. It is shown that the minimax risk for the estimation of the volume of $G\in\KK$ is of order $\DS n^{-\frac{d+3}{2d+2}}$, and this rate of convergence is attained by the explicit estimator given in \cite{Gayraud1997}. The estimation of the volume of $G$ is also tackled in \cite{BaldinReiss2016}, where the same Poissonization-depoissonization procedure as mentioned above is used in order to obtain an estimator of $|G|$ based on a dilation of the random polytope $\hat K_n$.

\subsection{Convex bodies and covering numbers} \label{Sec:CovNum}

Covering numbers provide a powerful tool to describe the complexity of a class. In empirical process theory, they are often used in order to bound the statistical performance of an estimator in expectation or with high probability, when the estimator is obtained by optimizing a criterion, such as the likelihood function.

 %As mentioned in Section \ref{Adapt} above for the estimation of the support of a uniform distribution, the random polytope $\hat K_n$ maximizes the likelihood over the class $\mathcal K_d$ and under the Nikodym metric, its risk is measured using the rescaled error $\DS \frac{|G\setminus\hat K_n|}{|G|}$. In this setup, the problem of maximizing the likelihood cannot reduce to empirical process theory, since the likelihood function is supported on a set that depends on the data. However, tools such as covering numbers can be borrowed from empirical process theory in order to prove deviation inequalities for $\hat K_n$.

Consider the problem of estimating the support $G$ of a uniform distribution, with $G\in\K$. Because the support of the likelihood function (see Section \ref{Adapt}) depends on the unknown parameter itself, it is not valid to take its logarithm and it cannot be approached through the lens of empirical process theory. However, tools such as covering numbers can still be borrowed from that theory in order to prove deviation inequalities for $\hat K_n$.

Without loss of generality, one can assume that $B(a,d^{-1})\subseteq G\subseteq B(0,1)$ for some $a\in B(0,1)$. This guarantees that $G\in\KK$, which is a bounded class of convex bodies, and that $|G|$ is uniformly bounded from below. This is due to John's theorem (e.g., see \cite{ball1992ellipsoids}) and affine equivariance of $\hat K_n$. John's theorem (e.g., see \cite{ball1992ellipsoids}) implies the existence an invertible affine transformation $T:\R^d\to\R^d$ and a point $a\in B(0,1)$ with $B(a,d^{-1})\subseteq TG\subseteq B(0,1)$. Moreover, if we rather denote by $\hat K_n(X_1,\ldots,X_n)$ the convex hull of $X_1,\ldots,X_n$, then, $\hat K_n(X_1,\ldots,X_n)=T^{-1}\hat K_n(TX_1,\ldots,TX_n)$. Since $X_1,\ldots,X_n$ are i.i.d. uniform random points in $G$, $TX_1,\ldots,TX_n$ are i.i.d. uniform random points in $TG$, and $\DS \frac{|G\setminus\hat K_n(X_1,\ldots,X_n)|}{|G|}=\frac{|TG\setminus\hat K_n(TX_1,\ldots,TX_n)|}{|TG|}$. As a consequence, the rescaled risk $\DS \frac{|G\setminus\hat K_n|}{|G|}$ is bounded from above by $\DS \frac{|G\setminus\hat K_n|}{\beta_d}$ and we only need to bound $|G\setminus\hat K_n|$ uniformly on $\KK$ instead of the whole unbounded class $\K$.

Let $\varepsilon>0$ and let $\textsf{d}(\cdot,\cdot)$ be a metric on $\KK$ (e.g., Nikodym or Hausdorff distance). An $\varepsilon$-net of $\KK$ with respect to the metric $\textsf{d}(\cdot,\cdot)$ is a set $\mathcal N\subseteq \KK$ such that for all $G\in\KK$, there is $G^*\in\mathcal N$ with $\textsf{d}(G,G^*)\leq\varepsilon$. The $\varepsilon$-covering number of $\KK$ with respect to $\textsf{d}(\cdot,\cdot)$ is the minimum cardinality of an $\varepsilon$-net of $\KK$. The following theorem is an upper bound for the $\varepsilon$-covering number of $\KK$ with respect to the Hausdorff distance. By \cite[Lemma 2]{Brunel2017}, the Nikodym distance is dominated by the Hausdorff distance uniformly on $\KK$: $\textsf{d}_\triangle(G_1,G_2)\leq\alpha \textsf{d}_{\textsf{H}}(G_1,G_2)$, for all $G_1,G_2\in\KK$, where $\alpha$ is a positive constant that depends on $d$ only. This result is a direct consequence of Steiner formula for convex bodies (see Lemma \ref{Steiner}). Hence, the following theorem also implies an upper bound for the $\varepsilon$-covering number of $\KK$ with respect to the Nikodym metric.

\begin{theorem}[\cite{Bronshtein1976}] \label{ThmBro}

Let $\varepsilon\in (0,1)$. The $\varepsilon$-covering number of $\KK$ with respect to the Hausdorff distance is at most $c_1e^{c_2\varepsilon^{-(d-1)/2}}$, for some positive constants $c_1$ and $c_2$ that depend on $d$.

\end{theorem}

We also refer to Section 8.4 in \cite{Dudley2014} for more details on metric entropy for classes of convex sets. Building on this theorem combined with standard techniques from M-estimation and empirical processes, (see, e.g., \cite{Vandervaart,Vandegeer}), \cite{Brunel2017} proves the following deviation inequality for $\hat K_n$, which holds uniformly for all $G\in\K$.

\begin{theorem}[\cite{Brunel2017}]

There exist positive constants $a_1, a_2$ and $a_3$ such that the following holds. Let $x\geq 0$ and $n\geq 2$ be an integer. For all $G\in\K$,
$$\frac{|G\setminus\hat K_n|}{|G|}\leq a_1 n^{-\frac{2}{d+1}}+\frac{x}{n}$$
with $\PP_G$-probability at least $1-a_2e^{-a_3x}$.

\end{theorem}

Using the same techniques, more general deviation inequalities are proven in \cite{Brunel2017}, when the density of the $X_i$'s is not uniform, but only supported on a convex body $G$. For all measurable sets $G,G'\subseteq \R^d$ and all densities $f$ on $\R^d$, denote by $\textsf{d}_f(G,G')=\int_{G\setminus G'}f(x)\diff x$. Note that $\DS \textsf{d}_f(G,G')=\frac{|G\setminus G'|}{|G|}$ when $f$ is the uniform density on $G$.

\begin{theorem}[\cite{Brunel2017}]
	There exist positive constants $C_1$ and $C_2$, that depend on $d$ only, such that the following holds. Let $x\geq 0$ and $n\geq 2$ be an integer. Let $G\in\KK$ and $f$ be a density supported in $G$, with $f\leq M$ almost everywhere, for some positive number $M$. Let $X_1,\ldots,X_n$ be i.i.d. random points with density $f$ and $\hat K_n$ be their convex hull. Then, 
\begin{equation*}
	\textsf{d}_f(G,\hat K_n)\leq C_1(M+1)n^{-2/(d+1)}+\frac{x}{n}
\end{equation*}
with probability at least $1-C_2e^{-x}$. 
\end{theorem}

It is not known whether a similar upper bound would hold without the assumption that $f\leq M$ almost everywhere. This open problem amounts to the following open question. Let $\mu$ be any probability measure supported in a convex body $G\in\KK$. Do there exist positive constants $c_1$ and $c_2$ that only depend on $d$, such that the $\varepsilon$-covering of $\KK$ with respect to the metric $d(G_1,G_2)=\mu(G_1\triangle G_2), G_1,G_2\in\KK$ is bounded from above by $c_1e^{c_2\varepsilon^{-(d-1)/2}}$, for all $\varepsilon\in (0,1)$? If $\mu$ has a bounded density $f$ with respect to the Lebesgue measure, the answer is positive, and it is a consequence of Theorem \ref{ThmDudley} below.

In the uniform case, concentration inequalities for $\hat K_n$ were proven in \cite{Vu2005}, using geometric techniques. However, constants were not explicit and depended on the support $G$, hence, could not be used in a minimax approach.

\subsection{Application of empirical process theory to the estimation of density level sets}

In this section, we show how similar ideas as in Section \ref{Sec:CovNum} can be used to estimate density level sets under a convexity restriction, in the Nikodym metric. The level sets of a density $f$ in $\R^d$ are the sets $G_\lambda=\{x\in\R^d:f(x)\geq\lambda\}$, for $\lambda> 0$. Estimation of density level sets and, more specifically, of convex level sets, has been tackled, e.g., in \cite{Hartigan1987, Polonik1995, Tsybakov1997}. As pointed by \cite{Hartigan1987}, estimation of density level sets may be useful in cluster analysis. It arises as a natural tool in testing for multimodality \cite{muller1991excess} and, more recently, it has been explored under the lens of topological data analysis \cite{wasserman2016topological,chazal2017introduction}. Notice that the $0$-level set of a density $f$ is its support, so support estimation is a particular case of density level set estimation. However, in this section we only treat the case of positive levels $\lambda$, where empirical process theory has proven to be a successful tool.

 Let $\lambda>0$ such that $G_\lambda\neq\emptyset$. The excess mass of a measurable set $C\subseteq \R^d$ is defined as $\mathcal M_\lambda(C)=\int_C f(x)\diff x-\lambda |C|$. Simple algebra shows that $\mathcal M_\lambda(C)\leq \mathcal M_\lambda(G_\lambda)$, for all measurable sets $C\subseteq \R^d$. The empirical excess mass of a set $C$, given a sample $X_1,\ldots,X_n$, is naturally defined as $\widehat {\mathcal M}_\lambda(C)=\frac{1}{n}\sum_{i=1}^n\mathds 1_{X_i\in C}-\lambda |C|$. Hence, the main idea to estimate $G_\lambda$ is to maximize $\widehat{\mathcal M}_\lambda(C)$ over $C\in\mathcal C$, where $\mathcal C$ is a given class of measurable sets. In this section, we assume that $G_\lambda\in\KK$ and we take $\mathcal C=\KK$. For instance, convexity of $G_\lambda$ is ensured if $f$ is log-concave or, more generally, quasiconcave. If $f$ is the uniform density on a convex body $G\in \KK$, then $G=G_\lambda$ for all $\lambda\in (0,\beta_d^{-1})$: In that case, support estimation is equivalent to level set estimation, for small levels $\lambda$ and the methods presented here could be applied to estimate $G$ itself. In what follows, $\lambda>0$ is a fixed number and we define the estimator $\hat G_n\in \underset{G\in\KK}{\operatorname{argmax}}\mbox{ }\widehat{\mathcal M}_\lambda(G)$.

In order to achieve consistency, an assumption is usually made about the behavior of $f$ around the boundary of its level set $G_\lambda$. Namely, $f$ should not be too flat near the boundary of $G_\lambda$. The assumption proposed in \cite{Polonik1995} takes the following form, where $\mu$ is the continuous probability measure on $\R^d$ with density $f$.

\begin{assumption} \label{AssumpLevelSet}
	There exist positive constants $c$ and $\gamma$ such that
	$$\mu\left(\{x\in\R^d:|f(x)-\lambda|<\eta \}\right)\leq c\eta^\gamma,$$
	for all $\eta>0$ small enough. 
\end{assumption}

Assumption \ref{AssumpLevelSet}, also known as \textit{margin condition}, is usually imposed for discriminant analysis \cite{MammenTsybakov1999,LoustauMarteau2015}, statistical learning \cite{Tsybakov2004}, level set estimation (a stronger assumption is proposed in \cite{Tsybakov1997}, see Assumption \ref{AssumpLevelSet2} below) or density support estimation \cite{Brunel2017}.

In \cite{Polonik1995}, the notion of covering number with inclusion, slightly different from that of covering number, is used to prove the main results.

\begin{definition}[Covering number with inclusion]

Let $\mathcal C$ be a class of measurable subsets of $B(0,1)$, $\mu$ a probability distribution in $\R^d$ and $\varepsilon>0$. The $\varepsilon$-covering number of $\mathcal C$ with inclusion with respect to $\mu$ is the smallest integer $N$ such that there exists a collection $\mathcal N$ of measurable sets, with $\#\mathcal N=N$, satisfying the following: For all $C\in\mathcal C$, there exist $C_*,C^*\in\mathcal N$ with $C_*\subseteq C\subseteq C^*$ and $\mu(C^*\setminus C_*)\leq\varepsilon$. It is denoted by $N_I(\varepsilon,\mathcal C,\mu)$ and $\ln N_I(\varepsilon,\mathcal C,\mu)$ is called the \textit{metric entropy with inclusion} of the class $\mathcal C$ with respect to $\mu$.

\end{definition}

Note that in this definition, $\mathcal N$ need not be included in $\mathcal C$. Also note that a similar notion, called metric entropy with bracketing, is widely used in function estimation, especially in empirical process theory (e.g., see \cite[Section 19.2]{Vandervaart}). Let $(\mathcal E,\|\cdot\|)$ be a normed space of real-valued functions defined on a set $X$ and let $\mathcal F\subseteq \mathcal E$. For any two functions $l,r\in \mathcal E$, the bracket $[l,r]$ is defined as the set of all functions $f\in \mathcal F$ satisfying $l(x)\leq f(x)\leq r(x)$ for all $x\in X$. For all $\varepsilon>0$, the $\varepsilon$-bracketing number of $F$ with respect to $\|\cdot\|$ is the smallest numbers of brackets $[l,r]$ with $\|r-l\|\leq\varepsilon$ needed to cover $\mathcal F$. It is denoted by $N_{[]}(\varepsilon,\mathcal F,\|\cdot\|)$ and $\ln N_{[]}(\varepsilon,\mathcal F,\|\cdot\|)$ is called the \textit{metric entropy with bracketing} of the class $\mathcal F$ with respect to $\|\cdot\|$. It is easy to see that for all class of measurable sets $\mathcal C$, if we let $\mathcal F_{\mathcal C}=\{\mathds 1_C:C\in\mathcal C\}$, then $\DS \ln N_I(\varepsilon,\mathcal C,\mu)$ and $\DS \ln N_{[]}(\varepsilon,\mathcal F_{\mathcal C},\|\cdot\|_{1,\mu})$ differ by at most a factor $2$, where $\DS \|\phi\|_{1,\mu}=\int_{\R^d}|\phi(x)|\diff\mu(x)$, for all measurable, bounded functions $\phi:\R^d\to\R$. The following estimate is available for the class $\KK$:

\begin{theorem}[\cite{Dudley2014}] \label{ThmDudley}
Let $\mu$ be a continuous probability measure on $B(0,1)$ with a density $f$ with respect to the Lebesgue measure. Assume that $f\leq M$ almost everywhere, where $M>0$ is a given number. Then, as $\varepsilon\to 0$,
	$$\ln N_I(\varepsilon,\KK,\mu)\lesssim_{d,M} \varepsilon^{-\frac{d-1}{2}}.$$ 
\end{theorem}

Together with this estimate, \cite[Theorem 3.7]{Polonik1995} yields the following result.

\begin{theorem}[\cite{Polonik1995}] \label{ThmPol95}
Assume that $d\geq 2$. There exists a constant $c(d)$ such that the following holds with probability tending to one, as $n$ goes to infinity. Let $\mu$ be a probability measure on $B(0,1)$ with a bounded density $f$ with respect to the Lebesgue measure and let Assumption \ref{AssumpLevelSet} hold. Let $\lambda>0$ and let $G_\lambda\in\KK$. Then,
$$\mu\left(\hat G_n\triangle G_\lambda\right)\leq \begin{cases} c(2)n^{-\frac{2\gamma}{3\gamma+4}} \mbox{ if } d=2, \\ c(3)n^{-\frac{\gamma}{2\gamma+2}}\ln n \mbox{ if } d=3 \\ c(d)n^{-\frac{2\gamma}{(\gamma+1)(d+1)}} \mbox{ if } d\geq 4. \end{cases}$$
\end{theorem}

In fact, this theorem is stated under more general assumptions than convexity of the level sets. If the level sets belong to a class of sets with metric entropy with inclusion of order $\varepsilon^{-r}$, for some exponent $r>0$ (e.g., $r=(d-1)/2$ for the class $\KK$), the rates given in the theorem depend on $d$, $\gamma$ and $r$. It is noticeable that the exponent $r=(d-1)/2$ in the metric entropy with inclusion of the class $\KK$ matches that of the class of sets with twice differentiable boundaries in some sense (see \cite{Dudley2014}). 

Note that the estimator $\hat G_n$ defined in \cite{Polonik1995} is a polytope, and its vertices are sample points. Indeed, for all $C\in\KK$, $\widehat{\mathcal M}_\lambda(C)\leq \widehat{\mathcal M}_\lambda(C^*)$, where $C^*$ is the convex hull of the sample points contained in $C$. In the two dimensional case, \cite{Hartigan1987} designs an algorithm to compute $\hat G_n$. To the best of our knowledge, there is no algorithm to compute $\hat G_n$ or an approximation of $\hat G_n$ in higher dimensions. 

Optimality of the upper bounds in the above theorem is not proven in \cite{Polonik1995}. However, \cite{Tsybakov1997} proves lower bounds for the minimax risk in both Nikodym and Hausdorff metrics. In the Nikodym metric, the lower bounds proven by \cite{Tsybakov1997} match the upper bounds given in the above theorem only for $d=2,3$ (up to a logarithmic factor when $d=3$), and they are faster when $d\geq 4$. The estimation of convex level sets in the Hausdorff metric requires completely different techniques. It has been tackled in \cite{sager1979iterative,Tsybakov1997}. In \cite{sager1979iterative}, the author considers both level sets corresponding to a given level and level sets with given probability content (see also \cite{Cadre} for the estimation of level sets with given probability content); The results are then applied to the estimation of the mode of the density, by considering the smallest estimated level set. Note that a control of the estimated level sets in the Nikodym metric could not yield consistent estimation of the mode, since two sets can have a very small Nikodym distance if they both have very small volumes, even if they are far apart from each other in the space. Optimal rates in estimation of convex density level sets in both Nikodym and Hausdorff metrics are given in \cite{Tsybakov1997} when $d=2$ and they are extended to higher dimensions. More generally, \cite{Tsybakov1997} proves optimal rates for density level sets whose boundaries satisfy some smoothness condition. In fact, it is noticed that if $G\in\K$ satisfies $B(0,r)\subseteq G\subseteq B(0,R)$ for some $0<r<R$, then the boundary of $G$ is Lipschitz, in the sense that the radial function of $G$, defined as $r_G(u)=\max\{\lambda\geq 0:\lambda u\in G\}$, is Lipschitz. For completeness, we include the precise statement and its proof here.

\begin{lemma} \label{LemmaLip}
	Let $G\in\K$ satisfies $B(0,r)\subseteq G\subseteq B(0,R)$ for some $0<r<R$. Then, the radial function $r_G$ satisfies $|r_G(u)-r_G(u')|\leq R/(2r)\|u-u'\|$, for all $u,u'\in\Sp$.
\end{lemma}

\begin{proof}
	Let $G^{\circ}$ be the polar body of $G$, defined as $G^{\circ}=\{x\in\R^d: \langle x,y\rangle\leq 1,\forall y\in K\}$. By standard properties of polar bodies (see \cite[Chapter 1]{Schneider1993}), one has $B(0,R^{-1})\subseteq G^{\circ}\subseteq B(0,r^{-1})$ and the radial function $r_G$ is the inverse of the support function of $G^{\circ}$: $\DS r_G(u)=\left(h_{G^{\circ}}(u)\right)^{-1}$, for all $u\in\Sp$. Subadditivity of support functions yield $\DS |h_{G^{\circ}}(u)-h_{G^{\circ}}(u')|\leq \max\left(h_{G^{\circ}}(u-u'),h_{G^{\circ}}(u'-u)\right)=\|u-u'\|\max\left(h_{G^{\circ}}\left(\frac{u-u'}{\|u-u'\|}\right),h_{G^{\circ}}\left(\frac{u'-u}{\|u-u'\|}\right)\right)$, for all $u,u'\in\Sp$ with $u\neq u'$. Since $G^{\circ}\subseteq B(0,r^{-1})$, $h_{G^{\circ}}(v)\leq r^{-1}$, for all $v\in\Sp$. This proves that $h_{G^{\circ}}$ is $r^{-1}$-Lipschitz. Now, since we also have that $B(0,R^{-1})\subseteq G^{\circ}$, $h_{G^{\circ}}(v)\geq R^{-1}$, for all $v\in\Sp$. Hence, $(h_{G^{\circ}})^{-1}$ is $(R/(2r))$-Lipschitz.
\end{proof}

First, \cite{Tsybakov1997} computes the optimal rates for star shaped density level sets with smooth radial functions. Standard techniques from functional estimation are used, such as local polynomial approximations. Then, the author tackles the problem of estimating convex level sets. As shown in Lemma \ref{LemmaLip}, the case of convex level sets is included in the case of star shaped level sets with Lipschitz radial functions. Hence, the optimal rates for convex sets are not larger than the ones corresponding to Lipschitz radial functions. Perhaps surprisingly, in the Hausdorff metric, convexity of the level set does not make the problem easier than just the Lipschitz property of its radial function, since \cite{Tsybakov1997} shows that the optimal rate under convexity matches the optimal rate under just the Lipschitz assumption, up to logarithmic factors. In the Nikodym metric, the situation is very different: \cite{Tsybakov1997} proves that at least in dimension $2$, the optimal rate for convex sets is actually much faster than in the case of Lipschitz radial functions: This is a consequence of Theorem \ref{ThmPol95} above. It can be seen easily that the same holds when $d=3$, and \cite{Tsybakov1997} suggests that this holds in arbitrary dimension, without a giving proof. Hence, in the Nikodym metric, convexity does contribute and improve the optimal rate from the Lipschitz assumption.

\cite{Tsybakov1997} does not exactly use the same margin condition as \cite{Polonik1995}, but makes the following assumption. Let $f$ be a density in $\R^d$ and let $G_\lambda$ be its level set with level $\lambda>0$. Assume that $G_\lambda$ is star shaped around the origin, and let $r_\lambda$ be its radial function.

\begin{assumption} \label{AssumpLevelSet2}

Let $b_1,b_2>0$ with $b_1<b_2$, $\nu,\delta_0>0$. Then, for all $u\in\Sp$ and $r>0$ such that $|f(ru)-\lambda|\leq\delta_0$,	
$$b_1\leq \frac{|f(ru)-\lambda|}{|r-r_{\lambda}(u)|^\nu}\leq b_2.$$

\end{assumption}

Roughly, Assumption \ref{AssumpLevelSet2} is stronger than Assumption \ref{AssumpLevelSet} if one takes $\gamma=1/\nu$. Under Assumption \ref{AssumpLevelSet2}, \cite{Tsybakov1997} characterizes the optimal rates for the estimation of a convex level set $G_\lambda$ that satisfies $B(0,r)\subseteq G_\lambda\subseteq B(0,R)$ with $0<r<R$ when $d=2$ and suggest the following extensions to higher dimensions: $n^{-2/(4\nu+d+1)}$ in the Nikodym metric and $n^{-1/(2\nu+d)}$ (up to a logarithmic factor) in the Hausdorff metric. In the Nikodym metric, the upper bound follows directly from \cite{Polonik1995} when $d=2$ but \cite{Tsybakov1997} does not give a proof for larger $d$. For arbitrary $d$, the rates suggested in \cite{Tsybakov1997} are actually faster than the upper bounds given in \cite{Polonik1995}. In the Hausdorff metric and for any $d$, as explained above, the upper bound follows directly from the Lipschitz case, by Lemma \ref{LemmaLip}. 

When dealing with level sets with smooth radial functions in arbitrary dimension, \cite{Tsybakov1997} proves that the minimax rates are exactly given by $n^{-\beta/((2\nu+1)\beta+d-1)}$ in the Nikodym metric and $(n/\ln n)^{-\beta/((2\nu+1)\beta+d-1)}$ in the Hausdorff metric, where $\beta$ is a smoothness parameter that roughly corresponds to the number of bounded derivatives of the radial function (e.g., $\beta=1$ corresponds to the Lipschitz case). It is noticeable that for convex level sets, the minimax rate $n^{-1/(2\nu+d)}$ in the Hausdorff metric matches the one that corresponds to smoothness $\beta=1$, as discussed above (and as predicted by Lemma \ref{LemmaLip}), whereas in the Nikodym metric, the minimax rate $n^{-2/(4\nu+d+1)}$ for convex level sets matches the rate that corresponds to smoothness $\beta=2$. This complements the remark we made earlier: The exponent $r=(d-1)/2$ in the metric entropy with inclusion for convex bodies is the same as for sets with twice differentiable boundary (see \cite{Polonik1995} and \cite{Dudley2014} for more details), and the corresponding minimax rates match. However, note that even though the boundary of any convex body is twice differentiable almost everywhere, the class of convex bodies $G\in\K$ with $B(0,r)\subseteq G\subseteq B(0,R)$, where $0<r<R$, contains polytopes with arbitrarily many vertices, which have very non-smooth boundaries, together with convex bodies with smooth boundaries and positive curvature everywhere, which yet can take arbitrarily large values.

Finally, note that the rates $n^{-2/(4\nu+d+1)}$ and $n^{-1/(2\nu+d)}$ given in \cite{Tsybakov1997} match (up to logarithmic factors) those obtained in the estimation of the support of a uniform distribution, i.e., at the limit $\nu=0$. In the Nikodym metric, the minimax rate of estimation of convex bodies is $n^{-2/(d+1)}$ (see Theorem \ref{ThmMinimaxSuppEst} above), whereas in the Hausdorff metric, it is $(n/\ln n)^{-1/d}$, as shown in Theorem \ref{ThmHausdorff2017} (with $\alpha=d$), see \cite{BrunelHausdorff2017}.

\subsection{Convex support estimation in nonparametric regression}

Let the following model hold:
$$Y_i=f(X_i)+\xi_i, \quad i=1,\ldots,n,$$
where $X_1,\ldots,X_n$ are deterministic or random points in $[0,1]^d$, $\xi_1,\ldots,\xi_n$ are i.i.d. random variables, with mean zero, independent of $X_1,\ldots,X_n$ and $f:[0,1]\to [0,\infty)$. In this section, we are interested in the estimation of the support $G$ of $f$, i.e., the closure of the set $\{x\in [0,1]^d:f(x)>0\}$. Throughout the section, we assume that $G$ is a convex body included in $[0,1]^d$. In \cite{Brunel2013}, the function $f$ is the indicator function of $G$: $f(x)=1$ for $x\in G$, $f(x)=0$ otherwise. The design points $X_1,\ldots,X_n$ are i.i.d., uniformly distributed in $[0,1]^d$ and the $\xi_i$'s are sub-Gaussian, i.e.,
$\DS \E\left[e^{t\xi_1}\right]\leq e^{\frac{\sigma^2 t^2}{2}}$, for all $t\in\R$, where $\sigma>0$ need not be known. \cite{Brunel2013} considers a least squares estimator $\DS \hat G_n\in \underset{C\in\mathcal N}{\operatorname{argmin}}\mbox{ } \mathcal A(C)$, where $\mathcal N$ is a $\DS n^{-2/(d+1)}$-net of $\KK$, with respect to the Nikodym metric and $\DS \mathcal A(C)=\sum_{i=1}^n (1-2Y_i)\mathds 1_{X_i\in C}$. The following upper bound is shown in \cite{Brunel2013}, in which $\PP_G$ stands for the joint distribution of the sample with $\DS f(\cdot)=\mathds 1_{\cdot\in G}$.

\begin{theorem}[\cite{Brunel2013}]
	There exist three positive constants $C_1, C_2, C_3$ that depend on $d$ and $\sigma^2$ and a positive integer $n_0$ that depends on $d$ only, such that the following holds:
	
	For all $G\in\KK$, all $n\geq n_0$ and all $x\geq 0$,
	$$|\hat G_n\triangle G|\leq C_1 n^{-\frac{2}{d+1}}+\frac{x}{n}$$
	with $\PP_G$ probability at least $\DS 1-C_2e^{-C_3 x}$.
\end{theorem}

Of course, the estimator $\hat G_n$ is not computable in practice, since $\varepsilon$-nets of $\KK$ are not available. We believe that maximizing the functional $\mathcal A$ over the whole class $\KK$ would yield a similar upper bound, whose rate is proven to be minimax optimal. However, it would remain unclear how to compute the resulting estimator $\tilde G_n$. Nonetheless, note that for all $C\in\KK$, $\mathcal A(C)=\mathcal A(C^*)$, where $C^*$ is the convex hull of all the design points $X_i\in C$, implying that $\tilde G_n$ can be chosen to be a polytope whose vertices are design points. Another open question is whether $\hat G_n$ (or $\tilde G_n$) is adaptive to polytopal supports. In \cite{Brunel2013}, it is shown that the minimax rate on the class $\mathcal P_r$ of polytopes with at most vertices is of the order $(\ln n)/n$: Like in density support estimation, is the error of $\hat G_n$ (or $\tilde G_n$) of that order when the true support $G$ is a polytope, up to logarithmic factors?

Here, we have only discussed the case when $f$ is an indicator function, but more general models for $f$ are considered in \cite{KTlectureNotes1993}. All these models, though, impose a sharp separation condition on $f$, e.g., boundedness away from zero on its support, which essentially reduces to the case of indicator functions. To the best of our knowledge, harder cases, e.g., when $f$ satisfies a margin type condition, i.e., $\DS \left|\{x\in [0,1]^d:f(x)\leq \eta\} \right|\leq c\eta^\gamma$ for all $\eta>0$ small enough, where $c$ and $\gamma$ are positive constants, have not been tackled in the literature.

\section{Estimation of convex sets under the Hausdorff metric} \label{Sec:SupportFunction}

\subsection{Support functions and polyhedral representations of convex bodies} \label{Sec:PropSuppFunc}

Support functions play a central role in estimation of convex bodies under the Hausdorff metric. Indeed, the Hausdorff distance between two convex bodies $G_1$ and $G_2$ can be written in terms of their support functions $h_{G_1}$ and $h_{G_2}$:
\begin{lemma} \label{HausdorffSuppFunc}
For all convex bodies $G_1,G_2\in\K$, 
$$\textsf{d}_{\textsf{H}}(G_1,G_2)=\sup_{u\in\Sp}|h_{G_1}(u)-h_{G_2}(u)|.$$
\end{lemma}
Here, we state a few results about support functions that are useful in estimation of convex sets, and we refer to \cite{Schneider1993} for more details on their account. 

Note that a convex body is completely determined by its support function and $G=\{x\in\R^d:\langle u,x\rangle\leq h_G(u), \forall u\in\Sp\}$, for all $G\in\K$. 

A function $h:\Sp\to\R$ is the support function of a convex set if and only if it is subadditive, in the following sense. 

\begin{definition}
	Let $h:\Sp\to\R$. Define the function $\tilde h$ as $\tilde  h(v)=\|v\|h(v/\|v\|)$ if $v\neq 0$, $\tilde h(0)=0$. We say that $h$ is subadditive if $\tilde h$ is convex. 
\end{definition}

If $h:\Sp\to\R$ is subadditive, then it is the support function of the convex set $\{x\in\R^d:\langle u,x\rangle\leq h(u),\forall u\in\Sp\}$.  

A polyhedral representation of a convex body is a way of writing it as the intersection of closed halfspaces or, equivalently, as a collection of affine constraints. For $\phi:\Sp\to\R$, we let $G_\phi=\{x\in\R^d:\langle u,x\rangle\leq\phi(u), \forall u\in\Sp\}$. It is easy to see that $h_{G_\phi}(u)\leq \phi(u)$, for all $u\in\Sp$. In general, the two functions are not equal, since $\phi$ is not necessarily subadditive. Subadditivity is actually a necessary and sufficient condition for $\phi$ to be a support function (see \cite[Proposition 1]{BrunelTDLS2017}):

\begin{lemma} \label{LemmaSubadd}
	Let $\phi:\Sp\to\R$. Then, $\phi=h_{G_\phi}$ if and only if $\phi$ is subadditive.
\end{lemma}

Actually, an interesting consequence of this result is the following.

\begin{lemma}
	Let $\phi:\Sp\to\R$. Then, $h_{G_\phi}$ is the largest subadditive function that is smaller or equal to $\phi$.
\end{lemma}

\begin{proof}
	Let $g:\Sp\to\R$ be a subadditive function with $g(u)\leq \phi(u)$, for all $u\in\Sp$. Then, $G_g\subseteq G_\phi$. As a consequence, $h_{G_g}(u)\leq h_{G_\phi}(u)$, for all $u\in\Sp$. By Lemma \ref{LemmaSubadd}, since $g$ is subadditive, $h_{G_g}=g$, yielding $g(u)\leq h_{G_\phi}(u)$ for all $u\in\Sp$. 
\end{proof}

By Lemma \ref{HausdorffSuppFunc}, the Hausdorff distance between two convex sets $G_\phi$ and $G_\psi$ can be written in terms of $h_{G_\phi}$ and $h_{G_\psi}$. However, these support functions may not be easy to compute in terms of $\phi$ and $\psi$. The following lemma provides a partial solution to this issue when both $\phi$ and $\psi$ are continuous. 

\begin{lemma}[\cite{BrunelTDLS2017}] \label{ThmTDLS1}

Let $\phi,\psi:\Sp\to\R$ be two continuous functions. Assume that $G_\phi$ and $G_{\psi}$ have nonempty interiors. Moreover, let $R>r>0$ and assume that $B(a,r)\subseteq G_\phi\subseteq B(a,R)$, for some $a\in\R^d$. Let $\eta=\max_{u\in\Sp}|\psi(u)-\phi(u)|$. If $\eta<r$, then $\DS \textsf{d}_{\textsf{H}}(G_{\psi},G_\phi)\leq \frac{\eta R}{r}\frac{1+\eta/r}{1-\eta/r}$.

\end{lemma}

In this lemma, the mapping $\psi$ plays the role of an estimate of $\phi$. Moreover, a control of the estimation error of $\phi$ in sup-norm yields a control of the estimation error of $G_{\phi}$ in the Hausdorff metric. However, in certain cases, it may not be easy to control the accuracy of the estimation of $\phi(u)$ for all $u\in\Sp$ simultaneously, but instead, only on a finite subset of $\Sp$. Recall that for $\varepsilon\in (0,1)$, an $\varepsilon$-net of $\Sp$ is a subset $\mathcal N$ of $\Sp$ such that for all $u\in\Sp$, there is $u^*\in\mathcal N$ with $\|u-u^*\|\leq\varepsilon$. For a subset $\mathcal N$ of $\Sp$ and a function $\psi:\mathcal N\to\R$, let $G_{\psi}^{\mathcal N}=\{x\in\R^d:\langle u,x\rangle\leq\psi(u), \forall u\in\mathcal N\}$. Then, the following result is complementary to the previous lemma, when $\phi=h_{G_\phi}$ (i.e., by Lemma \ref{LemmaSubadd}, when $\phi$ is subadditive) and when $\psi$ is only defined on a (fine enough) discretization of $\Sp$.

\begin{lemma}[\cite{BrunelTDLS2017}] \label{ThmTDLSdiscrete}

Let $G\in\K$ and $\psi:\Sp\to\R$. Let $\varepsilon\in (0,1)$ and $\mathcal N$ be an $\varepsilon$-net of $\Sp$. Let $R>r>0$ and assume that $B(a,r)\subseteq G\subseteq B(a,R)$, for some $a\in\R^d$. Let $\eta=\max_{u\in\mathcal N}|\psi(u)-h_G(u)|$. Then, if $\eta<r$,
$\DS \textsf{d}_{\textsf{H}}(G_{\psi}^{\mathcal N},G)\leq \frac{\eta R}{r}\frac{1+\eta/r}{1-\eta/r}+\frac{2R\varepsilon}{1-\varepsilon}$. 

\end{lemma}

Finally, the next lemma provides a polyhedral representation of the convex hull of a finite collection of points in the space. It is straightforward, given that a linear function defined on a convex body is necessarily maximized at an extreme point, but it yields a useful representation of random polytopes.

\begin{lemma} \label{PolyhedralCH}
	Let $x_1,\ldots,x_n\in\R^d$ and let $K_n$ be their convex hull. Then, 
	$$h_{K_n}(u)=\max_{1\leq i\leq n}\langle u,x_i\rangle, \quad \forall u\in\Sp.$$
	In particular, $\DS K_n=\{x\in\R^d:\langle u,x\rangle\leq \max_{1\leq i\leq n}\langle u,x_i\rangle, \forall u\in\Sp\}$. 
\end{lemma}

Of course, this polyhedral representation is not optimal: since $K_n$ is a polytope, Minkowski--Weyl theorem for polyhedra states that only a finite number of affine constraints should be sufficient to describe the set $K_n$. Moreover, these constraints correspond to the normal vectors of the $(d-1)$-dimensional faces of $K_n$: All other affine constraints are redundant. However, in practice, finding the $(d-1)$-dimensional faces of the convex hull of a given finite collection of points is a hard problem (see, e.g., \cite{BremnerFukudaMarzetta1998}).

\subsection{Estimation of support functions}

Sometimes, in order to estimate a convex body $G$, it can be natural to estimate its support function $h_G$. The previous lemmas indicate that if the estimation error are measured with respect to the Hausdorff metric, then it is enough to bound the pointwise or sup-norm error of the estimation of $h_G$. 

In this section, $G$ is an unknown convex body and we assume that an estimator $\hat h_n$ of $h_G$ has been extracted from available data. Then, $G_{\hat h_n}$ is a natural estimator of $G$. We review examples and show how the results and properties stated in the previous section apply, by distinguishing two cases: when $\hat h_n$ is subadditive, hence, when it is the support function $G_{\hat h_n}$, and when it is not.

\subsubsection{\bf When $\hat h_n$ is subadditive} 

\mbox{ }\vspace{2mm}

This is the case, for instance, in density support estimation: A natural estimator of the support $G$ of i.i.d. random points $X_1,\ldots,X_n$ is $\DS \hat h_n(u)=\max_{1\leq i\leq n}\langle u,X_i\rangle$, $u\in\Sp$. By Lemma \ref{PolyhedralCH}, $\hat h_n$ is actually the support function of $\hat K_n$, hence, it is subadditive.

\begin{lemma} \label{lemmaGeneralHausdorffTechnique}
Let $G\in\KK$. Let $\hat h_n$ be a subadditive estimator of $h_G$ and let $\hat G_n=G_{\hat h_n}$. 
Let $\delta,\eta\in (0,1)$ and assume that $|\hat h_n(u)-h_G(u)|\leq \eta$ with probability at least $1-\delta$, for all $u\in\Sp$. Then,
$\textsf{d}_{\textsf{H}}(\hat G_n,G)\leq 2\eta$ with probability at least $1-2(18/\eta)^d\delta$.
\end{lemma}

\begin{proof}

The proof of this lemma relies on two fundamental facts. First, by a standard volumetric argument, there exists an $\varepsilon$-net of $\Sp$ of cardinality at most $(3/\varepsilon)^d$, for all $\varepsilon\in (0,1)$. Second, the following result, which is Lemma 5.2 in \cite{FresenVitale2014}, is a very elegant tool to work with approximations of the unit sphere. 

\begin{lemma}[\cite{FresenVitale2014}] \label{LemmaFresen}
	Let $\varepsilon\in (0,1)$ and let $\mathcal N$ be an $\varepsilon$-net of $\Sp$. Then, for all $u\in\Sp$, there are sequences $\DS (u_k)_{k\geq 0}\subseteq\mathcal N$ and $\DS (\varepsilon_k)_{k\geq 1}\subseteq \R$ such that $\DS u=u_0+\sum_{k=1}^\infty \varepsilon_k u_k$, with $\DS 0\leq \varepsilon_k\leq\varepsilon^k, \quad \forall k\geq 1$.
\end{lemma}

	Let $\varepsilon=\eta/6$ and let $\mathcal N$ be an $\varepsilon$-net of $\Sp$. Denote by $-\mathcal N=\{-u:u\in\mathcal N\}$. Let $\mathcal A$ be the event when $|\hat h_n(u)-h_G(u)|\leq\eta$, simultaneously for all $u\in\mathcal N\cup (-\mathcal N)$. By a union bound,
\begin{equation*}
	\PP[\mathcal A] \geq 1-\sum_{u\in\mathcal N\cup (-\mathcal N)}\PP[|\hat h_n(u)-h_G(u)|\leq \eta] \geq 1-2(\#\mathcal N)\delta \geq 1-2(18/\eta)^d\delta.
\end{equation*}
Now, assume that $\mathcal A$ is satisfied and let $u\in\Sp$. Then, with the notation of Lemma \ref{LemmaFresen}, using subadditivity of support functions and the fact that $h_G\leq 1$ (since $G\subseteq B(0,1)$),
\begin{align*}
	\hat h_n(u) & \geq \hat h_n(u_0)-\sum_{k\geq 1} \varepsilon^k\hat h_n(-u_k) \geq h_G(u_0)-\eta-\sum_{k\geq 1} \varepsilon^k(h_G(-u_k)+\eta) \\
	& \geq h_G(u_0)-\eta-2\sum_{k\geq 1} \varepsilon^k = h_G(u_0)-\eta-\frac{2\varepsilon}{1-\varepsilon} \\
	& \geq h_G(u)-\sum_{k\geq 1}\varepsilon^k h_G(u_k)-\eta-\frac{2\varepsilon}{1-\varepsilon} \geq h_G(u)-\eta-\frac{3\varepsilon}{1-\varepsilon} \geq h_G(u)-2\eta
\end{align*}
and similarly, $\DS	h_G(u)\geq \hat h_n(u)-2\eta$. Thus, if $\mathcal A$ is satisfied, then $\textsf{d}_{\textsf{H}}(\hat G_n,G)=\sup_{u\in\Sp}|\hat h_n(u)-h_G(u)|\leq 2\eta$, which finishes the proof of the lemma.
\end{proof}

This technique is applied in \cite{BrunelHausdorff2017} for support estimation under smoothness conditions (see Assumption \ref{AssumptionHausdorff} below), where ideas from \cite{DumbgenWalther1996} are refined in order to obtain non-asymptotic deviation inequalities. The general framework considered in that work is the following. Let $G\in\KK$ and $\mu$ a probability measure on $\R^d$, satisfying the following assumption. 

\begin{assumption} \label{AssumptionHausdorff}
For all $u\in\R^d$ and all $t\in [0,r]$, $\mu(C_G(u,t))\geq Lt^\alpha$, where:
\begin{itemize}
	\item $r, L, \alpha$ are given positive numbers with $r<1$;
	\item $C_G(u,t)$ is the cap of $G$ in the direction $u$ and with height $t$, i.e., 
			$$C_G(u,t)=\{x\in G: \langle u,x\rangle \geq h_G(u)-t\}.$$
\end{itemize}
\end{assumption}

Particular cases of such pairs $(G,\mu)$ include uniform distributions on general convex bodies ($\alpha=d$) or smooth convex bodies ($\alpha=(d+1)/2$), uniform distributions on the boundary of smooth convex bodies ($\alpha=(d-1)/2$), linear projections of uniform distributions that are supported on higher dimensional smooth convex bodies, distributions with densities supported on a convex body with polynomial decay near the boundary, etc. 

It is easy to show that under Assumption \ref{AssumptionHausdorff}, for all $u\in\Sp$ and all $t\in [0,r]$, $|\hat h_n(u)-h_G(u)|\leq t$ with probability at least $1-e^{-Lnt^\alpha}$. Hence, the following theorem is shown, using Lemma \ref{lemmaGeneralHausdorffTechnique}. Set $\DS \tau_\alpha=\begin{cases} 1 \mbox{ if } \alpha\geq 1 \\ 2^{\alpha-1} \mbox{  if } 0<\alpha<1. \end{cases}$

\begin{theorem}[\cite{BrunelHausdorff2017}] \label{ThmHausdorff2017}
	Let $r\in (0,1)$ and $L,\alpha>0$. There exists a positive constant $C$ such that the following holds. Set $\DS a_n=\left(\frac{C\ln n}{n}\right)^\frac{1}{\alpha}$ and $\DS b_n=n^{\frac{-1}{\alpha}}$. Let $G\in\KK$ and $\mu$ be a probability measure on $\R^d$ and let Assumption \ref{AssumptionHausdorff} hold. Then, for all $x\geq 0$ such that $\displaystyle{a_n+b_n x\leq r}$,
	$$\textsf{d}_{\textsf{H}}(\hat K_n,G)\leq 2a_n + 2b_n x$$
with $\mu$-probability at least $\DS 1-12^d\exp\left(-C_\alpha L x^\alpha\right)$.
\end{theorem}

As a consequence, under Assumption \ref{AssumptionHausdorff}, $\hat K_n$ satisfies $\textsf{d}_{\textsf{H}}(\hat K_n,G)=O_{\PP}\left(\left(\frac{\ln n}{n}\right)^{-1/\alpha}\right)$.

\begin{remark}
\begin{itemize}

	\item \cite{BrunelHausdorff2017} uses a different technique in order to bound the error of $\hat K_n$ when $G$ is a polytope that satisfies a standardness condition. Using the technique shown above would yield the rate $\DS\left((\ln n)/n\right)^{1/d}$ ($\alpha=d$), which is suboptimal. Indeed, the following inequality is shown when $G$ is a polytope with at most $p$ vertices ($p\geq d+1$) and satisfies a $\nu$-standardness condition, for some $\nu\in (0,1)$:
	$$\PP_G\left[n^{1/d}\textsf{d}_{\textsf{H}}(K,\hat K_n)\geq x\right] \leq pe^{-\nu\beta_d x^d}, \forall x\geq 0.$$
		
	\item Surprisingly, the adaptive feature of $\hat K_n$ disappears under the Hausdorff metric. As discussed in Section \ref{Sec:RandomPolytopes}, $\hat K_n$ adapts to polytopal supports under the Nikodym metric, and its rate of convergence in that metric is the fastest when $G$ is a polytope. On the contrary, the rate of $\hat K_n$ in the Hausdorff metric is the worst when the support $G$ is a polytope.

\end{itemize}
\end{remark}

Another model where the same ideas could be applied has been considered in \cite{GardnerKiderlenMilanfar2006,Guntuboyina2012}. Given independent observations $Y_i=h_G(u_i)+\varepsilon_i$, where $u_i\in\mathbb S^{d-1}$ is either deterministic or random and $\varepsilon_i$ is a zero mean error term, independent of $u_i$, \cite{Guntuboyina2012} estimates $h_G$ by the least squares estimator on the class of all subadditive functions $h:\Sp\to\R$. This produces an estimator $\hat h_n$ that is a support function itself. In general, $L^2$-type metrics are natural to measure the performance of least squares estimators. \cite{GardnerKiderlenMilanfar2006,Guntuboyina2012} prove rates of convergence of their estimator with respect to the $L^2$-distance, with minimax optimality proven in \cite{Guntuboyina2012}. However, the $L^2$-distance between the support functions of convex bodies does not translate into a natural and geometric measurement of a distance between the convex bodies themselves. \cite{GardnerKiderlenMilanfar2006} uses elegant norm inequalities for subadditive functions, which show that the $L^\infty$ distance is dominated by the $L^2$ distance in some sense (see \cite[Proposition 2.3.1]{Groemer1996} and \cite[Proposition 2.2]{GardnerKiderlenMilanfar2006}). In turn, they obtain error bounds for the $L^\infty$ metric between the support functions, and hence, for the Hausdorff distance between the convex sets. Nevertheless, as remarked in \cite{Guntuboyina2012}, these error bounds are very loose, and the optimal rate of estimation of $G\in\KK$ with respect to the Hausdorff distance in that model is still an unsolved problem.

\subsubsection{\bf When the estimator $\hat h_n$ is not subadditive} \label{Sec:TukeyAndStuff}

\mbox{ }\vspace{2mm}

There are many problems where an estimator of an unknown convex body $G$ is defined as $\hat G_n=G_{\hat h_n}$ for some estimator $\hat h_n$ of $h_G$, that is not necessarily subadditive. Lemmas \ref{ThmTDLS1} and \ref{ThmTDLSdiscrete} are two key results to deal with this case. Lemma \ref{ThmTDLS1} is useful when $\hat h_n$ is continuous almost surely and its error can be bounded uniformly on the unit sphere $\Sp$. Indeed, if $|\hat h_n(u)-h_G(u)|\leq \eta$ simultaneously for all $u\in\Sp$ with high probability, then, if $G$ satisfies the assumptions of Lemma \ref{ThmTDLS1} and $\eta$ is small enough, $\textsf{d}_{\textsf{H}}(G_{\hat h_n},G)\leq\eta'$ with high probability, where $\eta'$ is a small number that depends on $\eta$ and other parameters. We apply Lemma \ref{ThmTDLSdiscrete} in two cases: when the true support function $h_G$ cannot be described simply, i.e., when $G=G_\phi$ for some function $\phi$ that is not necessarily subadditive, and when the error of $\hat h_n$ can only be controlled pointwise, i.e., for each $u\in\Sp$ separately. By a union bound, pointwise deviations transfer to uniform deviations on $\varepsilon$-nets of $\Sp$. We give two examples, borrowed from \cite{BrunelTDLS2017} and \cite{BrunelKlusowski2017}, where these results are used. \vspace{2mm}

\paragraph{Multivariate quantile estimation: level sets of the Tukey depth} \cite{BrunelTDLS2017}
\mbox{ }\vspace{2mm}

As already pointed out in \cite{fisher1966convex,fisher1969limiting,FresenVitale2014}, the study of the convex hull of a cloud of points is a multivariate extension of extreme value theory. Indeed, by Lemma \ref{PolyhedralCH}, the support function of $\hat K_n$ in any direction $u\in\Sp$ is given by the maximum of i.i.d. real random variables. In this regard, the definition of the random polytope $\hat K_n$ can be extended to that of the $k$-hull of the sample, for $0\leq k\leq n-1$. For each direction $u\in\Sp$, instead of considering the extreme statistic $\DS \max_{i=1,\ldots,n}\langle u,X_i\rangle$, take the $k$-th order statistics $X_{(k)}(u)$, where $X_i(u)=\langle u,X_i\rangle$ and $X_{(1)}(u)\geq X_{(2)}(u)\geq\ldots\geq X_{(n)}(u)$ is the reordered list of $X_1(u),\ldots,X_n(u)$ in nonincreasing order. The \textit{$k$-hull} of $X_1,\ldots,X_n$ is the convex set $\hat K_n^{(k)}=\{x\in\R^d:\langle u,x\rangle \leq X_{(k)}(u), \forall u\in\Sp\}$ \cite{ColeSharirYap1987}. The convex hull is the $0$-hull. In \cite{BrunelTDLS2017}, $k$-hulls are used as estimators of the Tukey depth level sets of probability measures in $\R^d$.

Let $\mu$ be a probability measure on $\R^d$. The Tukey depth $D_\mu(x)$ of a point $x\in\R^d$ with respect to $\mu$ is defined as the smallest probability mass of a closed halfspace containing $x$: 
\begin{equation*}
	D_\mu(x)=\inf_{x\in H\in\mathcal H_d} \mu(H),
\end{equation*}
where $\mathcal H_d$ is the collection of closed half-spaces in $\R^d$. For $\alpha\in (0,1)$, the $\alpha$-level set of $D_\mu$ is the set $G=\{x\in\R^d:D_\mu(x)\geq\alpha\}$. It is a closed convex set and it has the following polyhedral representation (see \cite[Theorem 2]{KongMizera2012} and \cite[Lemma 1]{BrunelTDLS2017}):
\begin{equation} \label{PolyhTDLS}
	G=\{x\in\R^d: \langle u,x\rangle \leq q_u, \forall u\in\Sp\},
\end{equation}
where $q_u$ is the upper $(1-\alpha)$-quantile of $\langle u,X\rangle$ with $X\sim\mu$. Let $X_1,\ldots,X_n$ be i.i.d. random points with distribution $\mu$ and let $\mu_n$ be the corresponding empirical measure: $\DS \mu_n(A)=\frac{1}{n}\sum_{i=1}^n\mathds 1_{X_i\in A}$, for all Borel sets $A\subseteq \R^d$. The empirical $\alpha$-level set of the Tukey depth is $\hat G_n=\{x\in\R^d: D_{\mu_n}(x)\geq \alpha\}$, which also has a simple polyhedral representation: $\DS \hat G_n=\{x\in\R^d: \langle u,x\rangle \leq \hat q_u, \forall u\in\Sp\}$, where $\hat q_u=\sup\{t\in\R: \#\{i=1,\ldots, n: \langle u,X_i\rangle\geq t\}\geq n\alpha\}$ is the empirical upper $(1-\alpha)$-quantile of the sample $\langle u,X_1\rangle,\ldots,\langle u,X_n\rangle$, for all $u\in\Sp$. With the above notation, $\hat q_u= X_{(k)}(u)$, for $k=\lceil n\alpha\rceil$, where, for $x\in\R$, $\lceil x\rceil$ stands for the smallest integer larger or equal to $x$. Hence, the empirical level set $\hat G_n$ of the Tukey depth coincides with the $k$-hull of $X_1,\ldots,X_n$. At the population level, there is another very important connection to the theory of random polytopes: As shown in \cite{BrunelTDLS2017}, if $\mu$ is the uniform distribution on a convex body $K\in\K$, the population level set $G$ coincides with the $\alpha$-floating body of $K$, as defined in Section \ref{Subsec:RandomPol}. 

In general, neither $u\mapsto q_u$ nor $u\mapsto \hat q_u$ are subadditive, and hence, they are not the support functions of the sets $G$ and $\hat G_n$ respectively. However, $u\mapsto \hat q_u$ is always continuous, as a consequence of \cite[Lemma 15]{BrunelTDLS2017} and $u\mapsto q_u$ is continuous under some weak assumptions on $\mu$ (see \cite[Lemma 11]{BrunelTDLS2017}). Building on standard results from empirical process theory, \cite{BrunelTDLS2017} shows that under some assumption on $\mu$, $\sup_{u\in\Sp}|\hat q_u-q_u|$ is small, with high probability. Hence, using Lemma \ref{ThmTDLSdiscrete}, this yields the following result.

\begin{theorem} \cite[Corollary 2]{BrunelTDLS2017} \label{TheoremTDLSMain}

	Let $\mu$ have a density $f$ with respect to the Lebesgue measure and let $0<\alpha<\max_{x\in\R^d}D_\mu(x)$. Assume that $f$ satisfies one of the following conditions:
	\begin{enumerate}
	
		\item $f$ is continuous and positive everywhere and there exist $C>0$ and $\nu>d-1$ such that $|f(x)|\leq C(1+\|x\|)^{-\nu}, \quad \forall x\in\R^d$.
		
		\item $f$ is supported on a bounded convex set and is uniformly continuous on its support.
	
	\end{enumerate}
	
	Then, $\DS \textsf{d}_{\textsf{H}}(\hat G_n,G)=O_{\PP}\left(n^{-1/2}\right)$.

\end{theorem}

More precise deviation inequalities are proven in \cite{BrunelTDLS2017} but here, we only state the result in this form for simplicity of the development. Note that any log-concave distribution satisfies one of the two conditions imposed on $f$ in this theorem and, as a particular case, the uniform distribution on a convex body $K\in\K$ satisfies the second condition. Computational questions are tackled in Section \ref{Sec:CurseOfDimension}. \vspace{2mm}

\paragraph{Estimation of the convex support of a density from noisy observations} \cite{BrunelKlusowski2017}
\mbox{ }\vspace{2mm}

The last model that we present here is an extension to density support estimation, where the observations are contaminated with some additive noise. Consider a convex body $G\in\K$ and let $X_1,\ldots,X_n$ be i.i.d. random points uniformly distributed in $G$. Assume that only $Y_i=X_i+\xi_i, i=1,\ldots,n$, are observed, where the $\xi_i$'s are i.i.d. centered Gaussian vectors with covariance matrix $\sigma^2 I_d$ ($I_d$ is the identity matrix and $\sigma^2>0$ is known), independent of the $X_i$'s. For $u\in\Sp$, a consistent estimator of $h_G(u)$ is given by $\hat h_n(u)=\max_{1\leq i\leq n} \langle u,Y_i\rangle -b_n$, where $b_n\approx \sqrt{2\sigma^2\ln n}$ is a deterministic debiasing term, and the set estimator of $G$ is defined as $\hat G_n=G_{\hat h_n}$. By extending some results on extreme value statistics \cite{Goldenshluger2004}, \cite{BrunelKlusowski2017} proves that $|\hat h_n(u)-h_G(u)|$ is small with high probability, for every single $u\in\Sp$.  However, due to the constant term $b_n$ in its definition, $\hat h_n$ is not subadditive in general. In \cite{BrunelKlusowski2017}, Lemma \ref{ThmTDLSdiscrete} is used in order to show that 
$$\textsf{d}_{\textsf{H}}(\hat G_n,G)=O_{\PP}\left(\frac{\ln \ln n}{\sqrt{\ln n}}\right).$$

A stronger deviation inequality is proven in \cite{BrunelKlusowski2017}, which uses the same ideas described above. Moreover, building on techniques developped in \cite{Wasserman2012}, it is also shown that the minimax rate of estimation of $G$ in this model is very slow and that $\hat G_n$ is nearly rate optimal in a minimax sense.

\section{Approximating convex set estimators: A statistical/computational trade-off} \label{Sec:CurseOfDimension}

The computational complexity of a set estimator can be very large, and some estimators may not even be computable in practice. For instance, computing the convex hull of $n$ points in $\R^d$ requires approximately $n^{d/2}$ operations (see, e.g., \cite{Chazelle1993}, and \cite{barber1996quickhull} for a description of the \textit{Quickhull} algorithm, implemented by the function \texttt{convhulln} in the \textbf{R} package \texttt{geometry}), which is not doable if the dimension is too large. Note, though, that in many applications that require the estimation of a set, the dimension is typically small: In spatial data analysis such as home range estimation, convex hulls or local convex hulls \cite{getz2004local}, or $\alpha$-hulls \cite{burgman2003bias} are computed with $d=2$; In econometrics, the efficient boundary problem deals with the feasible productivity domain of a firm, and $d$ is roughly the number of inputs at the firm \cite{simar2000statistical}; In image reconstruction, e.g., from satellite data \cite{Rodriguez2016}, $d=2$, typically. As for the empirical level sets of the Tukey depth, aka $k$-hulls \cite{ColeSharirYap1987}, there is no available algorithm to compute them in high dimensions. As we have seen above (Lemma \ref{PolyhedralCH} and \eqref{PolyhTDLS}), $k$-hulls (including convex hulls, for $k=0$) can be easily written using infinitely many affine constraints. When $k=0$, those constraints define a subadditive function, which turns out to be the support function of the random polytope, by Lemma \ref{LemmaSubadd}. n Section \ref{Sec:TukeyAndStuff}, when $k\geq 1$, we saw that in general, the function defined by these constraints is not subadditive, hence, the support function of the $k$-hull is not simply determined by the univariate order statistics $X_{(k)}(u), u\in\Sp$.

The idea exploited in \cite{BrunelHausdorff2017} for convex hulls and in \cite{BrunelTDLS2017} for $k$-hulls ($k\geq 1$) is to select a finite number of these constraints, leading to a larger approximating polytope, using a discretization of the unit sphere. One must bear in mind the purpose of approximating these sets, which is the estimation of an underlying set: The support of a density, or a population level set of the Tukey depth. Hence, an acceptable approximation error may be as large as the statistical error of the initial estimator. 

In order to discretize $\Sp$, the idea is to sample enough independent random unit vectors. The following lemma, proven in \cite{BrunelTDLS2017}, shows that with high probability this procedure provides an $\varepsilon$-net of $\Sp$, where $\varepsilon$ depends on the number of sampled unit vectors.

\begin{lemma}[\cite{BrunelTDLS2017}]

	Let $M$ be a positive integer and $U_1,\ldots,U_M$ be i.i.d. uniform vectors in $\Sp$. Let $\varepsilon\in (0,1)$ and let $\mathcal C$ be the event satisfied when the collection $\{U_1,\ldots,U_M\}$ is an $\varepsilon$-net of $\Sp$. Then, 
$$\PP[\mathcal C]\geq 1-6^d\exp\left(-\frac{M\varepsilon^{d-1}}{2d8^{\frac{d-1}{2}}}+d\ln\left(\frac{1}{\varepsilon}\right)\right).$$

\end{lemma}

Together with Lemma \ref{ThmTDLSdiscrete}, this result is used in \cite{BrunelHausdorff2017} in order to compute an approximate convex hull in $\DS O\left(d^28^{\frac{d}{2}}n^{2-1/d}(\ln n)^{1/d}\right)$ steps if the data are uniformly distributed in a general convex body and $\DS O\left(d^28^{\frac{d}{2}}n^{\frac{3d-1}{d+1}}(\ln n)^{-\frac{d-3}{d+1}}\right)$ steps if the data are uniformly distributed in a smooth convex body. In \cite{BrunelTDLS2017}, an approximate $k$-hull is computed in $O\left(8^{\frac{d}{2}}d^3n^{d}\ln n\right)$ operations, under some extra assumptions on the population level set, when $k=\alpha n$ for some fixed number $\alpha\in (0,1)$. Even though \cite{BrunelTDLS2017} provides an algorithm to approximate the $k$-hull, it still suffers the curse of dimensionality, because of the factor $(\sqrt 8n)^d$ in its complexity.

\bibliographystyle{plain}
\bibliography{Biblio}

\begin{thebibliography}{GPPVW12}

\bibitem[Bà92]{Barany1992}
I.~Bàràny.
\newblock Random polytopes in smooth convex bodies.
\newblock {\em Mathematika}, 39:81--92, 1992.

\bibitem[Bal92]{ball1992ellipsoids}
K.~Ball.
\newblock Ellipsoids of maximal volume in convex bodies.
\newblock {\em Geometriae Dedicata}, 41(2):241--250, 1992.

\bibitem[BB93]{BaranyBuchta1993}
I.~B{\'a}r{\'a}ny and C.~Buchta.
\newblock Random polytopes in a convex polytope, independence of shape, and
  concentration of vertices.
\newblock {\em Mathematische Annalen}, 297(1):467--497, 1993.

\bibitem[BDH96]{barber1996quickhull}
C.~B. Barber, D.~P. Dobkin, and H.~Huhdanpaa.
\newblock The quickhull algorithm for convex hulls.
\newblock {\em ACM Transactions on Mathematical Software (TOMS)},
  22(4):469--483, 1996.

\bibitem[Bel18]{bellec2018sharp}
P.~Bellec.
\newblock Sharp oracle inequalities for least squares estimators in shape
  restricted regression.
\newblock {\em The Annals of Statistics}, 46(2):745--780, 2018.

\bibitem[BF03]{burgman2003bias}
M.~A. Burgman and J.~C. Fox.
\newblock Bias in species range estimates from minimum convex polygons:
  implications for conservation and options for improved planning.
\newblock In {\em Animal Conservation forum}, volume~6, pages 19--28. Cambridge
  University Press, 2003.

\bibitem[BFM98]{BremnerFukudaMarzetta1998}
D.~Bremner, K.~Fukuda, and A.~Marzetta.
\newblock Primal-dual methods for vertex and facet enumeration.
\newblock {\em Discrete \& Computational Geometry}, 20(3):333--357, 1998.

\bibitem[BHH08]{BoroczkyHoffmannHug2008}
K.~J. Böröczky, L.~M. Hoffmann, and D.~Hug.
\newblock Expectation of intrinsic volumes of random polytopes.
\newblock {\em Periodica Mathematica Hungaricas}, 57:143--164, 2008.

\bibitem[BKY]{BrunelKlusowski2017}
V.-E. Brunel, J.~Klusowski, and D.~Yang.
\newblock Estimation of convex supports from noisy measurements.
\newblock {\em Preprint, submitted: arXiv:1804.09879}.

\bibitem[BL88]{BaranyLarman1988}
I.~Bàràny and D.~G. Larman.
\newblock Convex bodies, economic cap coverings, random polytopes.
\newblock {\em Mathematika}, 35:274--291, 1988.

\bibitem[Bla23]{Blaschke1923}
W.~Blaschke.
\newblock {\em Vorlesungen \"uber {Differentialgeometrue} {I}{I}}.
\newblock Springer-{Verlag}, 1923.

\bibitem[BR10]{barany2010poisson}
I.~B{\'a}r{\'a}ny and M.~Reitzner.
\newblock Poisson polytopes.
\newblock {\em The Annals of Probability}, pages 1507--1531, 2010.

\bibitem[BR16]{BaldinReiss2016}
N.~Baldin and M.~Rei{\ss}.
\newblock Unbiased estimation of the volume of a convex body.
\newblock {\em Stochastic Processes and their Applications},
  126(12):3716--3732, 2016.

\bibitem[Bro76]{Bronshtein1976}
E.~M. Bronshtein.
\newblock $\epsilon$-entropy of convex sets and functions.
\newblock {\em Siberian {Mathematical} {Journal}}, 17:303--398, 1976.

\bibitem[Bru]{Brunel2017}
V.-E. Brunel.
\newblock Uniform deviation and moment inequalities for random polytopes with
  general densities in arbitrary convex bodies.
\newblock {\em submitted,
  \href{https://arxiv.org/abs/1704.01620}{arXiv:1704.01620}}.

\bibitem[Bru13]{Brunel2013}
V.-E. Brunel.
\newblock Adaptive estimation of convex polytopes and convex sets from noisy
  data.
\newblock {\em Electronic Journal of Statistics}, 7:1301--1327, 2013.

\bibitem[Bru14]{BrunelThesis}
V.-E. Brunel.
\newblock Non parametric estimation of convex bodies and convex polytopes.
\newblock {\em
  \href{https://tel.archives-ouvertes.fr/tel-01066977/document}{PhD Thesis,
  Universit\'e Pierre et Marie Curie and University of Haifa}}, 2014.

\bibitem[Bru16]{Brunel2016}
V.-E. Brunel.
\newblock Adaptive estimation of convex and polytopal density support.
\newblock {\em Probability Theory and Related Fields}, 164(1-2):1--16, 2016.

\bibitem[Bru18a]{BrunelTDLS2017}
V.-E. Brunel.
\newblock Concentration of the empirical level sets of {T}ukey’s halfspace
  depth.
\newblock {\em To appear in Probability Theory and Related Fields}, 2018.

\bibitem[Bru18b]{BrunelHausdorff2017}
V.-E. Brunel.
\newblock Uniform behaviors of random polytopes under the hausdorff metric.
\newblock {\em To appear in Bernoulli}, 2018.

\bibitem[Buc05]{buchta2005identity}
C.~Buchta.
\newblock An identity relating moments of functionals of convex hulls.
\newblock {\em Discrete \& Computational Geometry}, 33(1):125--142, 2005.

\bibitem[CFPL12]{cuevas2012statistical}
A.~Cuevas, R.~Fraiman, and B.~Pateiro-L{\'o}pez.
\newblock On statistical properties of sets fulfilling rolling-type conditions.
\newblock {\em Advances in Applied Probability}, 44(2):311--329, 2012.

\bibitem[CGS15]{ChatterjeeAl2015}
S.~Chatterjee, A.~Guntuboyina, and B.~Sen.
\newblock On risk bounds in isotonic and other shape restricted regression
  problems.
\newblock {\em The Annals of Statistics}, 43(4):1774--1800, 2015.

\bibitem[Cha93]{Chazelle1993}
B.~Chazelle.
\newblock An optimal convex hull algorithm in any fixed dimension.
\newblock {\em Discrete Computational Geometry}, 10:377--409, 1993.

\bibitem[Che76]{Chevalier1976}
J.~Chevalier.
\newblock Estimation du support et du contour du support d’une loi de
  probabilit{\'e}.
\newblock {\em Annales de l'Institut Henri Poincar{\'e}, Sect. B (NS)},
  12(4):339--364, 1976.

\bibitem[CL15]{chatterjee2015adaptive}
S.~Chatterjee and J.~Lafferty.
\newblock Adaptive risk bounds in unimodal regression.
\newblock {\em arXiv preprint arXiv:1512.02956}, 2015.

\bibitem[CM17]{chazal2017introduction}
F.~Chazal and B.~Michel.
\newblock An introduction to topological data analysis: fundamental and
  practical aspects for data scientists.
\newblock {\em arXiv preprint arXiv:1710.04019}, 2017.

\bibitem[CPP13]{Cadre}
B.~Cadre, B.~Pelletier, and P.~Pudlo.
\newblock Estimation of density level sets with a given probability content.
\newblock {\em Journal of Nonparametric Statistics}, 25(1):261--272, 2013.

\bibitem[CSY84]{ColeSharirYap1987}
R.~Cole, M.~Sharir, and C.~K. Yap.
\newblock On k-hulls and related problems.
\newblock In {\em Proceedings of the sixteenth annual ACM symposium on Theory
  of computing}, pages 154--166. ACM, 1984.

\bibitem[Cue09]{Cuevas2009}
A.~Cuevas.
\newblock Set estimation : Another bridge between statistics and geometry.
\newblock {\em Sociedad de Estad\'istica e Investigaci\'on Operativa},
  25:71--85, 2009.

\bibitem[Dud14]{Dudley2014}
R.~M. Dudley.
\newblock {\em Uniform central limit theorems}, volume 142 of {\em Cambridge
  Studies in Advanced Mathematics}.
\newblock Cambridge University Press, New York, second edition, 2014.

\bibitem[Dup22]{Dupin1822}
C.~Dupin.
\newblock {\em Applications de g\'eom\'etrie et de m\'echanique \`a la marine}.
\newblock Ponts et Chauss\'ees, Paris, 1822.

\bibitem[DW80]{DevroyeWise1980}
L.~Devroye and G.~L. Wise.
\newblock Detection of abnormal behavior via nonparametric estimation of the
  support.
\newblock {\em SIAM Journal on Applied Mathematics}, 38(3):480--488, 1980.

\bibitem[DW96]{DumbgenWalther1996}
L.~D{\"u}mbgen and G.~Walther.
\newblock Rates of convergence for random approximations of convex sets.
\newblock {\em Advances in applied probability}, 28(2):384--393, 1996.

\bibitem[Efr65]{Efron1965}
B.~Efron.
\newblock The convex hull of a random set of points.
\newblock {\em Biometrika}, 52:331--343, 1965.

\bibitem[ES81]{efron1981jackknife}
B.~Efron and C.~Stein.
\newblock The jackknife estimate of variance.
\newblock {\em The Annals of Statistics}, pages 586--596, 1981.

\bibitem[Fis69]{fisher1969limiting}
L.~Fisher.
\newblock Limiting sets and convex hulls of samples from product measures.
\newblock {\em The Annals of Mathematical Statistics}, pages 1824--1832, 1969.

\bibitem[FJ66]{fisher1966convex}
L.~D. Fisher~Jr.
\newblock The convex hull of a sample.
\newblock {\em Bulletin of the American Mathematical Society}, 72(3):555--558,
  1966.

\bibitem[FV14]{FresenVitale2014}
D.~Fresen and R.~Vitale.
\newblock Concentration of random polytopes around the expected convex hull.
\newblock {\em Electronic Communications in Probability}, 19:1--8, 2014.

\bibitem[Gay97]{Gayraud1997}
G.~Gayraud.
\newblock Estimation of functionals of density support.
\newblock {\em Mathematical Methods of Statistics}, 6:26--47, 1997.

\bibitem[Gef64]{Geffroy1964}
J.~Geffroy.
\newblock Sur un problème d'estimation géométrique.
\newblock {\em Publications de l'Institut de Statistique des Universités de
  Paris}, 13:191--210, 1964.

\bibitem[GKM06]{GardnerKiderlenMilanfar2006}
R.~J. Gardner, M.~Kiderlen, and P.~Milanfar.
\newblock Convergence of algorithms for reconstructing convex bodies and
  directional measures.
\newblock {\em The Annals of Statistics}, 34(3):1331--1374, 2006.

\bibitem[GPPVW12]{Wasserman2012}
C.~R. Genovese, M.~Perone-Pacifico, I.~Verdinelli, and L.~Wasserman.
\newblock Manifold estimation and singular deconvolution under hausdorff loss.
\newblock {\em The Annals of Statistics}, 40(2):941--963, 2012.

\bibitem[Gro74]{Groemer1974}
H.~Groemer.
\newblock On the mean value of the volume of a random polytope in a convex set.
\newblock {\em Archiv der Mathematik}, 25(1):86--90, 1974.

\bibitem[Gro96]{Groemer1996}
H.~Groemer.
\newblock {\em Geometric applications of {F}ourier series and spherical
  harmonics}, volume~61 of {\em Encyclopedia of Mathematics and its
  Applications}.
\newblock Cambridge University Press, Cambridge, 1996.

\bibitem[GT04]{Goldenshluger2004}
A.~Goldenshluger and A.~Tsybakov.
\newblock Estimating the endpoint of a distribution in the presence of additive
  observation errors.
\newblock {\em Statistics \& probability letters}, 68(1):39--49, 2004.

\bibitem[Gun12]{Guntuboyina2012}
A.~Guntuboyina.
\newblock Optimal rates of convergence for convex set estimation from support
  functions.
\newblock {\em The Annals of Statistics}, 40(1):385--411, 2012.

\bibitem[GW04]{getz2004local}
W.~M. Getz and C.~C. Wilmers.
\newblock A local nearest-neighbor convex-hull construction of home ranges and
  utilization distributions.
\newblock {\em Ecography}, 27(4):489--505, 2004.

\bibitem[Har87]{Hartigan1987}
J.~A. Hartigan.
\newblock Estimation of a convex density contour in two dimensions.
\newblock {\em Journal of the American Statistical Association},
  82(397):267--270, 1987.

\bibitem[HW16]{han2016multivariate}
Q.~Han and J.~A. Wellner.
\newblock Multivariate convex regression: global risk bounds and adaptation.
\newblock {\em arXiv preprint arXiv:1601.06844}, 2016.

\bibitem[HWCS17]{han2017isotonic}
Q.~Han, T.~Wang, S.~Chatterjee, and R.~J. Samworth.
\newblock Isotonic regression in general dimensions.
\newblock {\em arXiv preprint arXiv:1708.09468}, 2017.

\bibitem[KGS17]{KimGuntuboyinaSamworth2016}
A.~K.~H. Kim, A.~Guntuboyina, and R.~J. Samworth.
\newblock Adaptation in log-concave density estimation.
\newblock {\em The Annals of Statistics}, to appear, 2017.

\bibitem[KM12]{KongMizera2012}
L.~Kong and I.~Mizera.
\newblock Quantile tomography: using quantiles with multivariate data.
\newblock {\em Statistica Sinica}, pages 1589--1610, 2012.

\bibitem[KS16]{kim2016global}
A.~K. Kim and R.~J. Samworth.
\newblock Global rates of convergence in log-concave density estimation.
\newblock {\em The Annals of Statistics}, 44(6):2756--2779, 2016.

\bibitem[KST95a]{KorostelevSimarTsybakov1995'}
A.~P. Korostelev, L.~Simar, and A.~B. Tsybakov.
\newblock Efficient estimation of monotone boundaries.
\newblock {\em The Annals of Statistics}, 23:476--489, 1995.

\bibitem[KST95b]{KorostelevSimarTsybakov1995}
A.~P. Korostelev, L.~Simar, and A.~B. Tsybakov.
\newblock On estimation of monotone and convex boundaries.
\newblock {\em Publications de l'Institut de Statistique de l'Université de
  Paris}, 39:3--18, 1995.

\bibitem[KT93a]{KorostelevTsybakov1993}
A.~P. Korostelev and A.~B. Tsybakov.
\newblock Estimation of the density support and its functionals (in {R}ussian).
\newblock {\em Problemy Peredachi Informatsii}, 29(1):3--18, 1993.

\bibitem[KT93b]{KTlectureNotes1993}
A.~P. Korostel\"ev and A.~B. Tsybakov.
\newblock {\em Minimax theory of image reconstruction}, volume~82 of {\em
  Lecture Notes in Statistics}.
\newblock Springer-Verlag, New York, 1993.

\bibitem[KT94]{KorostelevTsybakov1994}
A.~P. Korostelev and A.~B. Tsybakov.
\newblock Asymptotic efficiency in estimation of a convex set (in {R}ussian).
\newblock {\em Problems of Information Transmission}, 30:317--327, 1994.

\bibitem[LM15]{LoustauMarteau2015}
S.~Loustau and C.~Marteau.
\newblock Noisy discriminant analysis with boundary assumptions.
\newblock {\em Journal of Nonparametric Statistics}, 27(4):425--441, 2015.

\bibitem[ML93]{ManiLevitska1993}
P.~Mani-Levitska.
\newblock Characterizations of convex sets.
\newblock In {\em Handbook of Convex Geometry, Part A}, pages 19--41. Elsevier,
  1993.

\bibitem[Mol05]{Molchanov2005}
I.~Molchanov.
\newblock {\em Theory of Random Sets}.
\newblock Springer Verlag, 2005.

\bibitem[Moo84]{Moore1984}
M.~Moore.
\newblock On the estimation of a convex set.
\newblock {\em The Annals of Probability}, 12:1090--1099, 1984.

\bibitem[MS91]{muller1991excess}
D.~W. M{\"u}ller and G.~Sawitzki.
\newblock Excess mass estimates and tests for multimodality.
\newblock {\em Journal of the American Statistical Association},
  86(415):738--746, 1991.

\bibitem[MT95]{MammenTsybakov1995}
E.~Mammen and A.~B. Tsybakov.
\newblock Asymptotical minimax recovery of sets with smooth boundaries.
\newblock {\em The Annals of Statistics}, pages 502--524, 1995.

\bibitem[MT99]{MammenTsybakov1999}
E.~Mammen and A.~B. Tsybakov.
\newblock Smooth discrimination analysis.
\newblock {\em The Annals of Statistics}, 27(6):1808--1829, 1999.

\bibitem[Par11]{Pardon2011}
J.~Pardon.
\newblock Central limit theorems for random polygons in an arbitrary convex
  set.
\newblock {\em The Annals of Probability}, 39:881--903, 2011.

\bibitem[Par12]{Pardon2012}
J.~Pardon.
\newblock Central limit theorems for uniform random polygons.
\newblock {\em Journal of Theoretical Probability}, 25:823--833, 2012.

\bibitem[PL08]{Pateiro2008}
B.~Pateiro-Lopez.
\newblock Set estimation under convexity type restrictions.
\newblock {\em
  \href{http://eio.usc.es/pub/pateiro/files/THESIS\_BeatrizPateiroLopez.pdf}{PhD
  Thesis}}, 2008.

\bibitem[Pol95]{Polonik1995}
W.~Polonik.
\newblock Measuring mass concentrations and estimating density contour clusters
  - an excess mass approach.
\newblock {\em The Annals of Statistics}, 23:855--881, 1995.

\bibitem[RC07]{Rodriguez2007}
A.~Rodr{\'i}guez~Casal.
\newblock Set estimation under convexity type assumptions.
\newblock {\em Annales de l'Institut Henri Poincar\'e. Probabilit\'es et
  Statistiques}, 43(6):763--774, 2007.

\bibitem[RCSN16]{Rodriguez2016}
A.~Rodr\'iguez-Casal and P.~Saavedra-Nieves.
\newblock A fully data-driven method for estimating the shape of a point cloud.
\newblock {\em ESAIM: Probability and Statistics}, 20:332--348, 2016.

\bibitem[Rei03]{Reitzner2003}
M.~Reitzner.
\newblock Random polytopes and the {Efron}-{Stein} jackknife inequality.
\newblock {\em The Annals of Probability}, 31:2136--2166, 2003.

\bibitem[Rei04]{Reitzner2004}
M.~Reitzner.
\newblock Stochastical approximation of smooth convex bodies.
\newblock {\em Mathematika}, 51(1-2):11--29, 2004.

\bibitem[Rei05]{reitzner2005central}
M.~Reitzner.
\newblock Central limit theorems for random polytopes.
\newblock {\em Probability theory and related fields}, 133(4):483--507, 2005.

\bibitem[RR77]{ripley1977finding}
B.~Ripley and J.-P. Rasson.
\newblock Finding the edge of a poisson forest.
\newblock {\em Journal of Applied probability}, 14(3):483--491, 1977.

\bibitem[RS63]{RenyiSulanke1963}
A.~Rényi and R.~Sulanke.
\newblock Über die konvexe {Hülle} von $n$ zufällig gewählten {Punkten}.
\newblock {\em Zeitschrift für Wahrscheinlichkeitstheorie und verwandte
  Gebiete}, 2:75--84, 1963.

\bibitem[RS64]{RenyiSulanke1964}
A.~Rényi and R.~Sulanke.
\newblock Über die konvexe {Hülle} von $n$ zufällig gewählten {Punkten}.
  ii.
\newblock {\em Zeitschrift für Wahrscheinlichkeitstheorie und verwandte
  Gebiete}, 3:138--147, 1964.

\bibitem[Sag79]{sager1979iterative}
T.~W. Sager.
\newblock An iterative method for estimating a multivariate mode and isopleth.
\newblock {\em Journal of the American Statistical Association},
  74(366a):329--339, 1979.

\bibitem[Sch93]{Schneider1993}
R.~Schneider.
\newblock {\em Convex bodies: the {Brunn}-{Minkowski} theory}.
\newblock Cambridge University Press, 1993.

\bibitem[Sch94]{Schutt1994}
C.~Schütt.
\newblock Random polytopes and affine surface area.
\newblock {\em Mathematische Nachrichten}, 170:227--249, 1994.

\bibitem[SW90]{SchuttWerner1990}
C.~Schütt and E.~Werner.
\newblock The convex floating body.
\newblock {\em Mathematica Scandinavica}, 66:275--290, 1990.

\bibitem[SW00]{simar2000statistical}
L.~Simar and P.~W. Wilson.
\newblock Statistical inference in nonparametric frontier models: The state of
  the art.
\newblock {\em Journal of productivity analysis}, 13(1):49--78, 2000.

\bibitem[Th{\"a}08]{Thale2008}
C.~Th{\"a}le.
\newblock 50 years sets with positive reach - a survey.
\newblock {\em Surveys in Mathematics and its Applications}, 3:123--165, 2008.

\bibitem[Tsy94]{Tsybakov1994}
A.~B. Tsybakov.
\newblock Multidimensional change-point problems and boundary estimation.
\newblock In {\em Change-point problems ({S}outh {H}adley, {MA}, 1992)},
  volume~23 of {\em IMS Lecture Notes Monogr. Ser.}, pages 317--329. Inst.
  Math. Statist., Hayward, CA, 1994.

\bibitem[Tsy97]{Tsybakov1997}
A.~Tsybakov.
\newblock On nonparametric estimation of density level sets.
\newblock {\em The Annals of Statistics}, 25:948--969, 1997.

\bibitem[Tsy04]{Tsybakov2004}
A.~B. Tsybakov.
\newblock Optimal aggregation of classifiers in statistical learning.
\newblock {\em The Annals of Statistics}, 32(1):135--166, 2004.

\bibitem[Tsy09]{Tsybakov2009}
A.~B. Tsybakov.
\newblock {\em Introduction to nonparametric estimation}.
\newblock Springer, 2009.

\bibitem[VdG98]{Vandegeer}
S.~Van~de Geer.
\newblock {\em Empirical Processes in M-estimation}.
\newblock Cambridge University Press, 1998.

\bibitem[VdV00]{Vandervaart}
A.~W. Van~der Vaart.
\newblock {\em Asymptotic Statistics}.
\newblock Cambridge University Press, 2000.

\bibitem[Vu05]{Vu2005}
V.~H. Vu.
\newblock Sharp concentration of random polytopes.
\newblock {\em Geometric \& Functional Analysis GAFA}, 15:1284--1318, 2005.

\bibitem[Wal97]{Walther1997}
G.~Walther.
\newblock Granulometric smoothing.
\newblock {\em The Annals of Statistics}, pages 2273--2299, 1997.

\bibitem[Wal99]{Walther1999}
G.~Walther.
\newblock On a generalization of blaschke’s rolling theorem and the smoothing
  of surfaces.
\newblock {\em Mathematical methods in the applied sciences}, 22:301--316,
  1999.

\bibitem[Was16]{wasserman2016topological}
L.~Wasserman.
\newblock Topological data analysis.
\newblock {\em Annual Review of Statistics and Its Application}, (0), 2016.

\bibitem[Zha02]{zhang2002risk}
C.-H. Zhang.
\newblock Risk bounds in isotonic regression.
\newblock {\em The Annals of Statistics}, 30(2):528--555, 2002.

\bibitem[Zie95]{Ziegler1995}
G.~M. Ziegler.
\newblock {\em Lectures on polytopes}.
\newblock Springer, 1995.

\end{thebibliography}

\end{document}